\documentclass[letterpaper,11pt]{article}
\usepackage{color}
\usepackage[usenames,dvipsnames]{xcolor}
\usepackage[small]{titlesec}
\usepackage{theorem,amsmath,amssymb,amscd, verbatim, graphicx}
\usepackage[all]{xy}
\usepackage{booktabs}
\usepackage[hyperfootnotes=false, colorlinks, linkcolor={blue}, citecolor={magenta}, filecolor={blue}, urlcolor={blue}]{hyperref}
\usepackage[overload]{textcase} 
\theoremstyle{change}
\allowdisplaybreaks
\newcommand{\leftexp}[2]{{\vphantom{#2}}^{#1}%
      \kern-\scriptspace%
      {#2}}
\newcommand{\tst}{\textstyle}
\renewcommand{\Im}{\mathrm{Im}}
\newcommand{\pp}{{\p_{-}\text{-}\mathrm{fin}}}
\renewcommand{\Re}{\mathrm{Re}}
\renewcommand{\H}{\mathbb H}
\newcommand{\h}{\mathfrak h}

\newcommand{\A}{{\mathbb A}}
\renewcommand{\AA}{{\mathcal A}}
\renewcommand{\k}{{\mathbf k}}

\newcommand{\D}{{\mathcal D}}
\newcommand{\isomto}{\overset{\sim}{\longrightarrow}}

\newcommand{\T}{{\mathcal T}}

\newcommand{\Tr}{\mathrm{Tr}}
\newcommand{\Q}{{\mathbb Q}}
\newcommand{\Z}{{\mathbb Z}}

\newcommand{\R}{{\mathbb R}}
\newcommand{\C}{{\mathbb C}}
\newcommand{\bs}{\backslash}

\newcommand{\p}{\mathfrak p}
\newcommand{\f}{\mathfrak f}
\newcommand{\g}{\mathfrak g}
\newcommand{\U}{\mathcal U}

\renewcommand{\O}{{\mathcal O}}
\newcommand{\GL}{{\rm GL}}

\newcommand{\GSp}{{\rm GSp}}
\newcommand{\Sp}{{\rm Sp}}
\newcommand{\cond}{\mathrm{cond}}

\newcommand{\Aut}{{\rm Aut}}

\newcommand{\sgn}{{\rm sgn}}

\newcommand{\vol}{{\mathrm {vol}}}

\newcommand{\Gal}{{\rm Gal}}
\newcommand{\Ga}{{\mathcal{G}}}

\newcommand{\Hom}{{\rm Hom}}

\newcommand{\sym}{{\rm sym}}
\newcommand{\vl}{{\rm vol}}

\newcommand{\mat}[4]{{\setlength{\arraycolsep}{0.5mm}\left[
\begin{smallmatrix}#1&#2\\#3&#4\end{smallmatrix}\right]}}
\newcommand{\qed}{\hspace*{\fill}\rule{1ex}{1ex}}
\newcommand{\forget}[1]{}
\def\qdots{\mathinner{\mkern1mu\raise0pt\vbox{\kern7pt\hbox{.}}\mkern2mu
\raise3.4pt\hbox{.}\mkern2mu\raise7pt\hbox{.}\mkern1mu}}

\newenvironment{proof}{\vspace{0ex}\noindent{\it Proof.}\hspace{0.1em}}
	{\hfill\qed\vspace{1ex}}

\newtheorem{lemma}{Lemma.}[section]
\newtheorem{theorem}[lemma]{Theorem.}
\newtheorem{corollary}[lemma]{Corollary.}
\newtheorem{proposition}[lemma]{Proposition.}

\newtheorem{remark}[lemma]{Remark.}

\begin{document}

\bibliographystyle{plain}
\thispagestyle{empty}
\begin{center}
     {\bf\Large The special values of the standard $L$-functions for $\GSp_{2n} \times \GL_1$ }

 \vspace{3ex}
Shuji Horinaga, Ameya Pitale, Abhishek Saha, Ralf Schmidt
\end{center}

\begin{abstract}\noindent
We prove the expected algebraicity property for the critical values of character twists of the standard $L$-function associated to vector-valued holomorphic Siegel cusp forms of archimede\-an type $(k_1, k_2, \ldots, k_n)$, where $k_n \ge n+1$ and all $k_i$ are of the same parity.  For the proof, we use an explicit integral representation to reduce to arithmetic properties of differential operators on vector-valued nearly holomorphic Siegel cusp forms. We establish these properties via a representation-theoretic approach.
\end{abstract}

\tableofcontents

\section{Introduction} The arithmetic of special values of $L$-functions is of great interest  in modern number theory. A central problem here is to prove the algebraicity and $\Aut(\C)$-equivariance (up to suitable periods) of critical $L$-values. For classical cusp forms on the upper-half plane,  Shimura \cite{shi75, shi76, shimura77} and Manin \cite{Manin} were the first to study the arithmetic of their critical $L$-values in the 1970s.  In this paper, we focus on  twists of standard (degree $2n+1$) $L$-functions of vector-valued Siegel cusp forms of degree $n$; these correspond to $L$-functions $L(s, \Pi \boxtimes \chi)$ on  $\GSp_{2n}(\A_\Q) \times \GL_1(\A_\Q)$ such that $\Pi_\infty$ is a holomorphic discrete series representation.
The first results in this case were obtained over forty years ago with the works of Harris \cite{harsieg} and Sturm \cite{sturm} who (independently) proved the expected algebraicity results for $\Pi$ corresponding to a scalar-valued Siegel cusp form of full level and $\chi=1$. Subsequent works on the critical $L$-values of Siegel cusp forms by various authors \cite{boch1985, bocherer-schmidt,bouganis, kozima, miz,shibook2} have strengthened and extended these results in various directions.

Nonetheless, a proof of algebraicity of critical $L$-values for  holomorphic forms on $\GSp_{2n} \times \GL_{1}$ in full generality has not yet been achieved. The case $n=2$ has now been largely resolved by a recent result \cite[Theorem 1.1]{PSS17}. For general $n$, one can parameterize the possible archimedean types $\Pi_\infty$ by $n$-tuples of positive integers $(k_1, k_2, \ldots, k_n)$, where  $k_1\ge k_2 \ge \ldots\ge k_n$. If $\Pi$ is associated to a \emph{scalar-valued} Siegel cusp form, we have $k_1=k_2 = \ldots = k_n$; this case is now essentially solved by the succession of works cited above.  The major stumbling block is the general \emph{vector-valued} case, where the various $k_i$ may not be equal.  It was proved by Kozima \cite{kozima} that the expected algebraicity property holds for the critical values of the (untwisted) standard $L$-function of a vector-valued Siegel cusp form of full level and archimedean type $(k, \ell, \ell, \ldots, \ell)$ with $k, \ell$ even and $k \ge \ell \ge 2n+2$.  Here we prove the conjectured algebraicity property of critical $L$-values for more general archimedean types, general ramifications, and twists by characters.
 \begin{theorem}[Theorem \ref{t:globalthmarithmetic}]\label{t:mainintro}
  Let $\Pi$ be a cuspidal automorphic representation of $\GSp_{2n}(\A_\Q)$ whose archimedean component $\Pi_\infty$ is the holomorphic discrete series representation with highest weight $(k_1, k_2, \ldots, k_n)$, where the $k_i$ are integers of the same parity and $k_1\ge k_2 \ge \ldots\ge k_n\ge n+1$. Let $S$ be a finite set of places of $\Q$ including $\infty$ such that  $\Pi_p$ is unramified for $p\notin S$. Let $F$ be a nearly holomorphic Siegel cusp form of  scalar weight $k_1$ with Fourier coefficients lying in a CM field such that\footnote{We show in Section \ref{s:autrep} that an $F$ satisfying these properties always exists.} (the adelization of) $F$ generates an automorphic representation $\Pi_F$ whose local component $\Pi_{F,p}$ is twist-equivalent to $\Pi_p$ for all $p \notin S$. For each Dirichlet character $\chi$ such that $\chi_\infty = \sgn^{k_1}$ and each integer $r$ such that $1\le r \le k_n-n$, $r \equiv k_n-n \pmod{2}$, define
  $$
   D(\Pi, \chi, r; F) = \frac{L^S(r,\Pi \boxtimes \chi)}{i^{k_1}\pi^{nk_1+nr+r} \Ga(\chi)^{n+1} \langle F, F\rangle}.
  $$
  In the special case $r=1$ we also assume that $\chi^2 \neq 1$. Then $\sigma(D(\Pi, \chi, r; F)) = D({}^\sigma\!\Pi, {}^\sigma\!\chi, r; {}^\sigma\!F)$ for $\sigma \in \Aut(\C)$; in particular $D(\Pi, \chi, r; F)$ lies in a CM field.
\end{theorem}
Above, $\Ga(\chi)$ denotes the Gauss sum and the notation $L^S$ indicates that we omit the local $L$-factors corresponding to places in $S$.
\begin{remark}\label{rem:introcritical}The set of critical points  for $L(s,\Pi \boxtimes \chi)$ in the right-half plane are given by integers $r$ such that
\begin{equation}\label{critical-set}
\{1 \leq r \leq k_n-n : r \equiv k_n-n \pmod{2}\}.
\end{equation}
Thus, Theorem \ref{t:globalthmarithmetic} includes all the critical points in the right half-plane except in the special case that $k_n \equiv n+1 \pmod{2}$ and $\chi$ is quadratic, in which case our theorem cannot handle the critical point $s=1$. The reason for this omission is subtle, and is related to the fact that the normalization of the Eisenstein series corresponding to this point involves the factor $L(1, \chi^2)$ which has a pole when $\chi^2=1$. Consequently the required arithmetic results for the Eisenstein series are unavailable in this case. We also note that the critical points in the left half-plane are related to those in the right-half plane via the global functional equation (which is known by Theorem 62 of \cite{CFK}).
\end{remark}

Classically, the automorphic representation $\Pi$ in Theorem \ref{t:mainintro} arises from a  $V$-valued \emph{holomorphic} Siegel cusp form $G$ on $\H_n$ such that  $G(\gamma Z) = \rho(J(\gamma, Z))G(Z)$ for all $\gamma$ in some principal congruence subgroup of $\Sp_{2n}(\Z)$, where  $(\rho, V)$ is the finite-dimensional representation of $\GL_n(\C)$ with highest weight $(k_1, k_2, \ldots, k_n)$. (We refer to such $G$ as vector-valued holomorphic Siegel cusp forms of archimedean type $(k_1, k_2, \ldots, k_n)$.) The vector-valued holomorphic form $G$ is related to the scalar valued nearly holomorphic form $F$ via certain differential operators.

Theorem \ref{t:mainintro} is an instance of a reciprocity result for the critical $L$-values in the spirit of a famous conjecture due to Deligne \cite{deligneconj}. However, Deligne's conjecture is in the motivic world and it is a non-trivial problem to relate Deligne's motivic period to our period  $\langle F, F\rangle$ appearing in Theorem \ref{t:mainintro}. One way to observe the compatibility of our result with Deligne's conjecture
is via ratios of $L$-values (which eliminates the periods involved). In that direction,  Theorem \ref{t:mainintro} implies the following consequence of  Deligne's conjecture.
\begin{corollary}\label{cor:mainintro}
  Let  $k_1\ge k_2 \ge \ldots\ge k_n\ge n+1$ be integers, where all $k_i$ have the same parity. Let $\Pi_1$ and $\Pi_2$ be irreducible cuspidal automorphic representations of $\GSp_{2n}(\A_\Q)$ such that $\Pi_{1,\infty} \simeq \Pi_{2, \infty}$ is the holomorphic discrete series representation with highest weight $(k_1, k_2, \ldots, k_n)$. Suppose that for almost all primes $p$, the representations $\Pi_{1,p}$ and $\Pi_{2,p}$ are twist-equivalent,\footnote{This essentially means that the two representations are equivalent when restricted to $\Sp_{2n}(\Q_p)$.} i.e., there exists a character $\psi_p$ of $\Q_p^\times$ satisfying $\Pi_{1,p} \simeq \Pi_{2,p} \otimes (\psi_p\circ\mu_n)$, where $\mu_n$ is the multiplier homomorphism from $\GSp_{2n}(\Q_p) \rightarrow \Q_p^\times$. Let $S$ be a finite set of places including $\infty$ such that  $\Pi_{1,p}$ and $\Pi_{2,p}$ are unramified for $p\notin S$. Then, for primitive Dirichlet characters $\chi_1$, $\chi_2$ such that $\chi_{1,\infty} = \chi_{2,\infty} = \sgn^{k_1}$, and integers $r_1$, $r_2$ such that $2\le r_1, r_2 \le k_n-n$, $r_1 \equiv r_2 \equiv k_n-n \pmod{2}$,
  \begin{equation}\label{e:last}
    \frac{L^S(r_1,\Pi_1 \boxtimes \chi_1)}{\pi^{(n+1)(r_1 - r_2)}L^S(r_2,\Pi_2 \boxtimes \chi_2)} \quad \text{lies in a CM field}.
  \end{equation}
\end{corollary}
In the rest of this introduction we explain briefly the key ideas in our proof of Theorem \ref{t:mainintro}. The starting point here is an explicit integral representation (or pullback formula) proved in \cite[Theorem 6.4]{PSS17}. This formula roughly says that
\begin{equation}\label{e:formulaold} \tst
 \left\langle E\left(\mat{-}{}{}{Z_2}, \frac{r+n-k_1}2 \right), \  F\right \rangle \approx L^S(r,\Pi \boxtimes \chi) F(Z_2)
\end{equation} where $E(Z, s)$ is a certain Eisenstein series on $\H_{2n}$ of weight $k_1$, the function $F$ on $\H_n$ corresponds to a smooth modular form of weight $k_1$ associated to a particular choice of archimedean vector inside $\Pi_\infty$, the element $\mat{Z_1}{}{}{Z_2}$ of $\H_{2n}$ is obtained from the diagonal embedding of $\H_n \times \H_n$, the Petersson inner product $\langle \ , \ \rangle$ is taken with respect to the $Z_1$ variable, and the symbol $\approx$ indicates that the two sides are equal up to some (well-understood) explicit factors. The Eisenstein series $E(Z, s)$ arises from the degenerate principal series obtained by inducing the character $\mat{A}{X}{}{v\,^t\!A^{-1}} \mapsto  \chi(v^{-n} \det(A)) |v^{-n} \det(A)|^{2s + k}$ of the Siegel parabolic of $\GSp_{4n}$; we also refer the reader to \cite[(120)]{PSS17} for an explicit formula for $E(Z, s)$ (denoted $E_{k,N}^{\chi}(Z,s)$ there).

Using the shorthand $\k = (k_1, k_2, \ldots, k_n)$, we prove in Section \ref{s:nearholofull} that $F = \D_\k(G)$ where $G$ is the vector-valued holomorphic Siegel cusp form on $\H_n$ of archimedean type $\k$ mentioned earlier, and $\D_\k$ is a certain differential operator obtained from a Lie algebra element.


Via a linear algebra argument, the proof of Theorem \ref{t:mainintro} can now be reduced to the heart of this paper, which is to show the $\Aut(\C)$-equivariance of  $\D_\k$ when viewed as an operator on the space of nearly holomorphic cusp forms. This equivariance is a priori not clear, as this operator is defined abstractly. To achieve this, we re-interpret the space of nearly holomorphic modular forms of degree $n$ in the representation theoretic language, generalizing our work in the $n=2$ case done in \cite{PSS14}. After laying the necessary Lie-algebraic foundations in Section \ref{s:Lie}, we show in Section \ref{s:ident} that the space of (vector-valued) nearly holomorphic modular forms can be identified with the space of $\p_{-}$-finite automorphic forms, where $\p_{-}$ (see Section \ref{basic-lie-alg-sec}) is the span of the non-compact negative roots of $\Sp_{2n}(\R)$. We go on to define the operator $\D_\k$ representation-theoretically by choosing suitable Lie algebra elements; the crucial $\Aut(\C)$-equivariance property of this operator (Proposition \ref{p:uoper}) is proved via a careful arithmetic analysis of the Lie algebra. A novel aspect of our methodology is that we expand the domain of functions under consideration from nearly holomorphic modular forms to functions which have a Fourier expansion involving polynomials in $\Im(Z)^{\pm \frac12}$ (see Section \ref{s:rationalnice}). Along the way, we prove several new results concerning nearly holomorphic Siegel cusp forms and $\p_{-}$-finite automorphic forms, including a structure theorem (Proposition \ref{AA0nfindecomp2prop}) and a finiteness result for the dimension of the space of all nearly holomorphic cusp forms of given level and archimedean type (Proposition \ref{l:fact2}). These results shed new light on nearly holomorphic forms from the representation-theoretic point of view and should be of independent interest.

Our approach differs from previous works in the direction of Theorem \ref{t:mainintro} in the vector-valued setup such as \cite{kozima}. There one works directly with holomorphic vector-valued Siegel cusp forms and invokes vector-valued Eisenstein series, relying on arithmetic properties of differential operators on tensor products of vector-valued functions \cite{bochsatyama, ibu99}. The relevant pullback formula for that approach has been recently worked out in more generality by Liu \cite{Liu16, Liu19}. In contrast, our pullback formula \eqref{e:formulaold} involves only scalar-valued (nearly holomorphic) forms and our differential operators correspond to elements of the universal enveloping algebra of the Lie algebra of $\Sp_{2n}(\R)$. Thus, our proofs of the algebraicity theorems hinge  on understanding the arithmetic properties of nearly holomorphic cusp forms and of the relevant elements in the Lie algebra. A byproduct of our method is an explicit formula for the scalar-valued function $E(\mat{Z_1}{}{}{Z_2}, \frac{r+n-k}2)$ --- whose $p$-integrality properties were recently proved by us in \cite{PSS19} ---  as a bilinear sum  over  nearly holomorphic Siegel cusp forms, with the coefficients equal to critical values of $L$-functions (see \eqref{GkNdefeq}, \eqref{keyidentity}). We hope to pursue further arithmetic applications of this formula elsewhere.

\subsection*{Acknowledgements} We thank the anonymous referee for helpful comments and insightful questions which have significantly improved this paper. A.S. acknowledges the support of the Leverhulme Trust Research Project Grant RPG-2018-401.

\subsection{Notation}\label{s:notation}
For any ring $R$, we let $M_n(R)$ denote the ring of $n$-by-$n$ matrices over $R$. We let $M^\sym_n(R)$ denote the submodule of symmetric matrices. For a commutative ring $R$, we let
$$
 \GSp_{2n}(R)=\{g\in\GL_{2n}(R)\ | \ ^tgJ_ng=\mu_n(g)J_n,\:\mu_n(g)\in R^\times\},\qquad J_n=\mat{}{I_n}{-I_n}{}.
$$
The symplectic group $\Sp_{2n}(R)$ consists of those elements $g\in \GSp_{2n}(R)$ for which the multiplier $\mu_n(g)$ is $1$. We let $\Gamma_{2n}(N)\subset \Sp_{2n}(\Z)$ be the preimage of the identity element in $\Sp_{2n}(\Z/N\Z)$ under the natural surjection  $\Sp_{2n}(\Z) \rightarrow \Sp_{2n}(\Z/N\Z)$.

The Siegel upper half space of degree $n$ is defined by
  $$
   \H_n= \{ Z \in M_n(\C) \ | \ Z = {}^tZ,\:i( \overline{Z} - Z) \text{ is positive definite}\}.
  $$
For $g=
\left[\begin{smallmatrix} A&B\\ C&D \end{smallmatrix}\right] \in \Sp_{2n}(\R)$, $Z\in \H_n$, let $J(g,Z) = CZ + D.$ For a finite dimensional representation $(\rho, V)$ of $\GL_n(\C)$, a function $f \in C^\infty(\H_n, V)$, and $g \in \Sp_{2n}(\R)$, we define the function $f|_\rho g \in C^\infty(\H_n, V)$ by $(f|_\rho g)(Z) = \rho(J(g, Z))^{-1}f(gZ)$.

We let $\g_n=\mathfrak{sp}_{2n}(\R)$ be the Lie algebra of $\Sp_{2n}(\R)$ and $\g_{n,\C} =\mathfrak{sp}_{2n}(\C)$ the complexified Lie algebra.  We let $\mathcal{U}(\mathfrak{g}_{n,\C})$ denote the  universal enveloping algebra and let $\mathcal{Z}_n$ be its center. For all smooth functions $f : \Sp_{2n}(\R) \rightarrow V$ where $V$ is a complex vector space, and $X \in \g_n$, we define $(Xf)(g)=\frac{d}{dt}\big|_0f(g\exp(tX))$. This action is extended $\C$-linearly to $\g_{n,\C}$. Further, it is extended to all elements $X \in \mathcal{U}(\mathfrak{g}_{n,\C})$ in the usual manner. Unless there is a possibility of confusion, we will  omit the subscript $n$  and freely use the symbols $\g$, $\g_\C$, $\mathcal{U}(\mathfrak{g}_{\C})$, $\mathcal{Z}$, etc.

We let $\A=\A_\Q$ denote the ring of adeles of $\Q$. The symbol $\mathfrak{f}$ denotes the set of finite places (i.e., primes) of $\Q$ and the symbol $\A_\f$ denotes the finite adeles. We define automorphic forms and representations as in \cite{BorelJacquet1979}.  All our automorphic representations are over $\A$ and all our  $L$-functions are normalized so that the expected functional equation takes $s \mapsto 1-s$. All automorphic representations are assumed to be irreducible. Cuspidal automorphic representations are assumed to be unitary. Each cuspidal representation $\pi$ of $G(\A)$  is isomorphic to a restricted tensor product $\otimes\pi_v$, where $\pi_v$ is an irreducible, admissible, unitary representation of $G(\Q_v)$.
Given an automorphic representation $\pi$ and a set of places $S$ of $\Q$,  we let $L^S(s, \pi)=\prod_{v\notin S} L(s, \pi_v)$ be the associated global $L$-function where the local factors coming from the places in $S$ are omitted. For a positive integer $N$, we denote $L^{N}(s, \pi):= L^{S_N}(s, \pi)$ where $S_N$ consists of the primes dividing $N$ and the place $\infty$.

We say that a character $\chi=\prod\chi_p$ of $\Q^\times \bs
\A^\times$ is a \emph{primitive Dirichlet character} if $\chi$ is of finite order, or equivalently, if $\chi_\infty$ is trivial
on $\R_{>0}$. We let $\cond(\chi)$ denote the conductor of such a $\chi$, and we identify $\cond(\chi)$ with a positive integer.  A primitive Dirichlet character $\chi$ as defined above  gives rise to a homomorphism $\tilde{\chi}:(\Z/\cond(\chi)\Z)^\times\rightarrow \C^\times$, via the formula $\tilde{\chi}(a) = \prod_{p|\cond(\chi)} \chi_p^{-1}(a)$, and the association $\chi \mapsto \tilde{\chi}$ is a bijection between primitive Dirichlet characters in our sense and in the classical sense. We define the Gauss sum $\Ga(\chi)$ by
$\Ga(\chi) = \sum_{n \in  (\Z/\cond(\chi)\Z)^\times }\tilde{\chi}(n) e^{2 \pi i
n/\cond(\chi)}.$

For a complex representation $\pi$ of some group $H$ and an automorphism $\sigma$  of $\C$, there is a complex representation ${}^\sigma \pi$ of $H$ defined as follows. Let $V$ be the space of $\pi$ and let $V'$ be any vector space such that $t: V \rightarrow V'$ is a $\sigma$-linear isomorphism (that is, $t(v_1 +v_2) = t(v_1) + t(v_2)$ and $t(\lambda v) = \sigma(\lambda) t(v)$). We define the representation $( \leftexp{\sigma}{\pi}, V')$ by $\leftexp{\sigma}{\pi}(g) = t \circ \pi(g) \circ t^{-1}$. It can be shown easily that the representation $\leftexp{\sigma}{\pi}$ does not depend on the choice of $V'$ or $t$. We define $\Q(\pi)$ to be the fixed field of the set of all automorphisms $\sigma$ such that $\leftexp{\sigma}{\pi} \simeq \pi$.

\section{Lie algebra action and differential operators}\label{s:Lie}
Throughout this section, $n$ will be a fixed positive integer. In this section, we will obtain some information on the action of the Lie algebra of $\Sp_{2n}(\R)$ and the corresponding differential operators acting on functions on the Siegel upper half space. We will also explain the relation between functions on the Siegel upper half space and on the group $\Sp_{2n}(\R)$.

\subsection{Basics on the Lie algebra of \texorpdfstring{$\Sp_{2n}(\R)$}{}}\label{basic-lie-alg-sec}
Recall that
$\g$ is the Lie algebra of $\Sp_{2n}(\R)$. Let $K_\infty$ denote the maximal compact subgroup of $\Sp_{2n}(\R)$ consisting of matrices of the form $\mat{A}{B}{-B}{A}$. We identify $K_\infty$ with $U(n)$ via $\mat{A}{B}{-B}{A} \mapsto A+iB$. Let $\mathfrak k$ be the Lie algebra of $K_\infty$ and $\mathfrak k_\C$ denote the complexification of $\mathfrak k$. Then we have the Cartan decomposition
$$\g_\C = \mathfrak k_\C \oplus \p_\C,$$
for some subspace $\p_\C$ of $\g_\C$. The complex structure of $\Sp_{2n}(\R)/K_\infty$ determines the decomposition $\p_\C = \p_+ \oplus \p_-$ satisfying (see \cite[p.~245]{Sh90})
\begin{equation}\label{pkcommeq}
 [\mathfrak k_\C, \p_{\pm}] = \p_{\pm}, \quad [\p_+, \p_+] = [\p_-, \p_-] = 0, \quad [\p_+, \p_-] = \mathfrak k_\C.
\end{equation}
The explicit description of $\p_{\pm}$ is given in \cite[p.~260]{Sh90} as follows. Set $T :=M_n^\sym(\C)$. As in \cite[p.~260]{Sh90}, we define the $\C$-linear isomorphisms $\iota_{\pm}$ from $T$ to $\p_{\pm}$ by
\begin{equation}\label{iota-defn}\tst
 \iota_{\pm}(u)=\frac 14 \mat{\mp  iu}{u}{u}{\pm  iu} \in \p_{\pm} \quad \text{ for } u \in T.
\end{equation}
 Let $e_{i,j}$ be the $n \times n$ matrix whose $(i, j)^{\rm th}$ entry is $1$ and all others are $0$. For $1\le i, j \le n$, we define, as in \cite{Mau12},
\begin{equation}\label{BijEijdefeq}
 B_{i,j} = \mat{\frac 12(e_{i,j}-e_{j,i})}{\frac{-i}2(e_{i,j}+e_{j,i})}{\frac{i}2(e_{i,j}+e_{j,i})}{\frac 12(e_{i,j}-e_{j,i})},
 \quad E_{\pm, i, j} =  E_{\pm, j, i} = \mat{\frac 12(e_{i,j}+e_{j,i})}{\frac{\pm i}2(e_{i,j}+e_{j,i})}{\frac{\pm i}2(e_{i,j}+e_{j,i})}{\frac{-1}2(e_{i,j}+e_{j,i})}.
\end{equation}
Then $\{B_{i,j} : 1 \leq i, j \leq n\}$ is a basis for $\mathfrak k_\C$ and $\{E_{\pm, i, j} : 1 \leq i \le j \leq n\}$ is a basis for $\p_{\pm}$.

A Cartan subalgebra $\mathfrak{h}_\C$ of $\mathfrak{k}_\C$ (and of $\mathfrak{g}_\C$) is spanned by the $n$ elements $B_{i,i}$, $1\le i \le n$. If $\lambda$ is in the dual space $\mathfrak{h}_\C^*$, we identify $\lambda$ with the element $(\lambda(B_{1,1}),\lambda(B_{2,2}), \ldots, \lambda(B_{n,n}))$ of $\C^n$. In this way we identify  $\mathfrak{h}_\C^\ast$ with $\C^n$. If $\mathfrak{k}_\C$ acts on a space $V$, and $v\in V$ satisfies $B_{i,i}v=\lambda_i v$  for $\lambda=(\lambda_1, \lambda_2, \ldots \lambda_n) \in\C^n$, then we say that $v$ has \emph{weight} $\lambda$.

We let $\Lambda =\Z^n \subset \mathfrak{h}_\C^\ast$ be the weight lattice, consisting of the integral weights. Let $V$ be a finite-dimensional $\mathfrak{k}_\C$-module. Then this representation of $\mathfrak{k}_\C$ can be integrated to a representation of $K_\infty$ if and only if all occurring weights lie in $\Lambda$. The isomorphism classes of irreducible $\mathfrak{k}_\C$-modules, or the corresponding irreducible representations of $K_\infty$, are called \emph{$K_\infty$-types}.

We take a system of positive roots to be
\[
 \Phi^+=\{e_i - e_j \mid 1 \leq i < j \leq n\} \cup \{e_i + e_j \mid 1 \leq i \leq j \leq n\},
\]
where $e_i$ is the element of $\C^n$ with 1 in the $i$'th position and 0 everywhere else. (Concretely, $e_i(B_{j,j}) = \delta_{i,j}$). The positive compact roots are $\{e_i - e_j\mid1\leq i<j\leq n\}$, and the corresponding root vectors are $\{B_{i,j}\mid1\le i<j\le n\}$. Let $V$ be a $K_\infty$-type, which we think of as a $\mathfrak{k}_\C$-module. A non-zero vector $v\in V$ is called a \emph{highest weight vector} if
$$
 B_{i,j} v =0 \text{ for all } 1\le i<j\le n.
$$
Such a vector $v$ is unique up to scalars. Let $\k=(k_1, k_2, \ldots, k_n)$ be its weight. Then the $k_i$ are integers and $k_1 \ge k_2 \ge \ldots \ge k_n$; we say that $\k$ is the highest weight of $V$. We let $\Lambda^+ \subset \Lambda$ be the set of tuples $\k=(k_1, k_2, \ldots, k_n)$ such that each $k_i \in \Z$ and $k_1 \ge k_2 \ge \ldots \ge k_n$. Then the elements of $\Lambda^+$ parametrize the irreducible representations of $K_\infty$ via their highest weights, as above. We let $\Lambda^{++}$ consists of those $\k=(k_1, \ldots, k_n) \in \Lambda^+$ with $k_n\geq1$.

\subsection{Generators for the center of the universal enveloping algebra}
	
For notational convenience we define matrices $B$, $E_+$, $E_-$ with matrix valued entries as follows:
	\[
	B = (B_{k,\ell})_{k,\ell} \in M_{n}(M_{2n}(\C)), \qquad E_{\pm} = (E_{\pm,k,\ell})_{k,\ell} \in M_n^\sym(M_{2n}(\C)).
	\]
	Let $B^* = (B_{\ell,k})_{k,\ell}$ be the transpose of $B$. 	We consider \emph{words} in the ``letters" $B$, $B^*$, $E_+$ and $E_-$ satisfying the following five conditions\footnote{Conditions 1-4 do not apply to the  last letter of a word.}.
	\begin{enumerate}
	\item $E_+$ is followed by $E_-$ or $B^*$.
	\item $E_-$ is followed by $E_+$ or $B$.
	\item $B$ is followed by $E_+$ or $B$.
	\item $B^*$ is followed by $E_-$ or $B^*$.
	\item $E_+$ and $E_-$ occur with the same multiplicity.
	\end{enumerate}

For a word $w = X_1 \cdots X_m$, we denote by $\Tr(w)$ the trace as the $M_{2n}(\C)$-valued matrix, which we identify with an element of $\U(\g_\C)$. Let $L(w)$ equal the sum of the number of times $E_-B$ and $BE_+$ occur isolatedly in $w$ counted cyclicly. By isolatedly, we mean that the $E_-B$ and $BE_+$ concerned must not intersect each other, so, for e.g., $L(E_-BE_+B^\ast) = 1$ but $L(E_-BBE_+) = 2$. And cyclically means that we have to take into account that the trace is cyclically invariant, so, e.g., $L(E_+E_-BB) = L(E_-BBE_+) = 2$. The following theorem is the main result of \cite{Mau12}.

\begin{theorem}[Theorem 2.2 of \cite{Mau12}]\label{t:Mau12}
	For $r \in \Z_{>0}$, put
	\[
	D_{2r} = \sum_w (-1)^{L(w)} \Tr(w),
	\]
	where the sum is over all words $w$ of length $2r$ satisfying the conditions (1) to (5) above. Then the center $\mathcal{Z}$ of $\U(\g_\C)$ is generated by the $n$ elements $D_2,\ldots,D_{2n}$ as an algebra over $\C$.
\end{theorem}

For our purposes, we will only need the following corollaries of the theorem above.

\begin{corollary}\label{Center-gener-cor}
  Let $1 \le r \le n$, and let $D_{2r}$ be as in Theorem \ref{t:Mau12}. Then there exists an expression of the form
  $$
   D_{2r} = \sum_{i=1}^t c_i H^{(i)}_{1}\cdots H^{(i)}_{r_i} X^{(i)}_1\cdots X^{(i)}_{p_i} Y^{(i)}_1\cdots Y^{(i)}_{q_i} E^{(i)}_{+,1}\cdots E^{(i)}_{+,s_i} E^{(i)}_{-,1}\cdots E^{(i)}_{-,s_i}.
  $$
  Above, $t\in \Z_{>0}$, and for each $i$, we have $c_i \in \Z$, $p_i, q_i, r_i, s_i \in \Z_{\ge 0}$. Moreover, each $H^{(i)}_k$ is equal to $B_{a,a}$ for some $a$, each $X^{(i)}_k$ is equal to $B_{a,b}$ for some $a < b$, each $Y^{(i)}_k$ is equal to $B_{a,b}$ for some $a > b$,  each $E^{(i)}_{+,k}$ is equal to some $E_{+,a,b}$, and each $E^{(i)}_{-,k}$ is equal to some $E_{-,a,b}$.
\end{corollary}
\begin{proof} Using Theorem \ref{t:Mau12}, we can write $D_{2r}$ as a $\Z$-linear combination of words in $B_{*,*}$, $E_{+,*,*}$, and $E_{-,*,*}$ where there are an equal number of $E_{+,*,*}$, and $E_{-,*,*}$ in each word. We now use the Lie bracket relations (see Section 1 of \cite{Mau12}):
\[ [E_{+,i,j}, E_{+,k,\ell}] = 0, \quad [E_{-,i,j}, E_{-,k,\ell}] = 0,\]
\[[E_{+,i,j}, E_{-,k,\ell}] = \delta_{i,k}B_{j,\ell}+\delta_{j,\ell}B_{i,k}+\delta_{i,\ell}B_{j,k}+\delta_{j,k}B_{i,\ell} \]
\[[B_{i,j}, E_{+,k,\ell}] = \delta_{j,k}E_{+,i,\ell} + \delta_{j,\ell}E_{+,i,k}\]
\[[B_{i,j}, E_{-,k,\ell}] = -\delta_{i,k}E_{-,j,\ell} - \delta_{i,\ell}E_{-,j, k}\]
\[[B_{i,j}, B_{k,\ell}] = \delta_{j,k}B_{i,\ell} - \delta_{i,\ell}B_{k,j}\]
to make the following moves on each word. First we move all the $E_{-,*,*}$ elements one by one to the right so that they make a block at the end. Then we move all the $E_{+,*,*}$  to the right so that they end up just before the $E_{-,*,*}$ part. Note that these two steps do not affect the equality of the number of $E_{+,*,*}$, and $E_{-,*,*}$ elements. Then we move all the $B_{a,b}, a>b$ elements to the right of the other  $B_{*,*}$ elements so that they end up just before the $E_{+,*,*}$ part. (Note by the relations above that no new $E_{+,*,*}$ or $E_{-,*,*}$ elements are generated by this step.) Finally we move the $B_{a,b}, a<b$ elements to the right of the  $B_{a,a}$ elements. This brings each word to the desired form.
\end{proof}

\begin{corollary}\label{Center-gener-cor2}
There exists an expression of the form
  $$
   D_{2r} = \sum_{i=1}^t c_i  E^{(i)}_{+,1}\cdots E^{(i)}_{+,s_i} E^{(i)}_{-,1}\cdots E^{(i)}_{-,s_i} H^{(i)}_{1}\cdots H^{(i)}_{r_i} X^{(i)}_1\cdots X^{(i)}_{p_i} Y^{(i)}_1\cdots Y^{(i)}_{q_i},
  $$
where $c_i \in \Z$, and $H^{(i)}_k$, $X^{(i)}_k$, $B_{a,b}$, $Y^{(i)}_k$, $E^{(i)}_{+,k}$ and $E^{(i)}_{-,k}$ are elements of $\g$ with the same meanings as in Corollary \ref{Center-gener-cor}.
\end{corollary}
\begin{proof} This follows from the expression obtained in Corollary \ref{Center-gener-cor}, by moving the $B_{a,b}$ elements to the very right, using the Lie bracket relations.
\end{proof}
\subsection{Differential operators}
Recall that $\H_n$ is the Siegel upper half space of degree $n$. Let $(\rho, V)$ be a finite dimensional representation of $\GL_n(\C)$. Recall that $T = \{ u \in M_n(\C)\ |\ {}^tu = u\}$. Let $\{\epsilon_\nu\}_\nu$ be any $\R$-rational basis for $T$. For $u \in T$, write $u = \sum_\nu u_\nu \epsilon_\nu$ with $u_\nu \in \C$, and for $z \in \H_n$ write $z = \sum_\nu z_\nu \epsilon_\nu$ with $z_\nu \in \C$.

For a non-negative integer $e$, let $S_e(T, V)$ denote the vector space of all homogeneous polynomial maps $T\to V$ of degree $e$. 
 Note that $S_0(T, V) = V$. We can identify $S_e(T, V)$ with the symmetric elements in the space ${\rm Ml}_e(T, V)$ of $e$-multilinear maps from $T^e$ to $V$ (see Lemma 12.4 of \cite{shibook2}). We define two representations of $\GL_n(\C)$ denoted by $\rho \otimes \tau^e$ and $\rho \otimes \sigma^e$ on ${\rm Ml}_e(T, V)$ as follows. For $h \in {\rm Ml}_e(T, V), (u_1, \cdots, u_e) \in T^e$ and $a \in \GL_n(\C)$, define
\begin{align*}
 [(\rho \otimes \tau^e)(a)h](u_1, \cdots, u_e) &:= \rho(a) h({}^tau_1a, \cdots, {}^tau_ea),\\
 [(\rho \otimes \sigma^e)(a)h](u_1, \cdots, u_e) &:= \rho(a) h(a^{-1}u_1{}^ta^{-1}, \cdots, a^{-1}u_e{}^ta^{-1}).
\end{align*}
We will denote the restriction of $\rho \otimes \tau^e$ and $\rho \otimes \sigma^e$ to $S_e(T, V)$ also by the same notation.

Given $f \in C^\infty(\H_n, V)$ we can define the functions $Df, \bar{D} f, Cf, Ef$ in $C^\infty(\H_n, S_1(T,V))$ as follows.
\begin{align}
 \big((Df)(z)\big)(u) := \sum_{\nu} u_\nu \frac{\partial f}{\partial z_{\nu}}(z), &\qquad \big((\bar{D} f)(z)\big)(u) := \sum_{\nu} u_\nu \frac{\partial f}{\partial \bar{z}_{\nu}}(z),\\ \label{e:defCE}
 \big((Cf)(z)\big)(u) := 4\big((D f)(z)\big)(y u y), &\qquad \big((E f)(z)\big)(u) := 4\big((\bar{D} f)(z)\big)(yu y).
\end{align}
Here, $u = (u_\nu) \in T, z = (z_\nu) \in \H_n$ and $y = {\rm Im}(z)$. These are exactly the formulas defined in \cite[p.~92]{shibook2} with $\xi(z) = \eta(z) = 2y$ in our case. For $0 \leq e \in \Z$, we can define $D^ef$, $\bar{D}^ef$, $C^ef$ and $E^ef$ recursively, and these take values in ${\rm Ml}_e(T, V)$. For $f \in C^\infty(\H_n, V)$, let $\rho(\Xi)f\in C^\infty(\H_n, V)$ be the function defined by $z\mapsto\rho(2y)f(z)$. More generally, for $f \in C^\infty(\H_n, {\rm Ml}_e(T, V))$, let $(\rho \otimes \tau^e)(\Xi )f\in C^\infty(\H_n, {\rm Ml}_e(T, V))$ be the function defined by $z\mapsto(\rho \otimes \tau^e)(2y)(f(z))$. We then define $D^e_\rho f \in C^\infty(\H_n, S_e(T,V))$ by
\begin{equation}\label{e:DefDrho}
 D_{\rho}^e f = (\rho \otimes \tau^e)(\Xi )^{-1} (C^e (\rho (\Xi) f)).
\end{equation}
The following is Proposition 12.10 of \cite{shibook2}.
\begin{proposition}\label{DE-behav-prop}
\begin{enumerate}
\item We have $D_{\rho}^{e+1}  = D_{\rho \otimes \tau^e} D_\rho^e = D^e_{\rho \otimes \tau} D_\rho$.
\item 	For  $f \in C^\infty(\H_n, V)$ and $\alpha\in \Sp_{2n}(\R)$, define the slash action by
$$
 (f |_\rho \alpha)(z) := \rho(J(\alpha, z))^{-1} f(\alpha z).
$$
Then, for all $\alpha \in \Sp_{2n}(\R)$, we have
	\[
	D_{\rho}^e (f |_\rho \alpha) = (D_{\rho}^e f)|_{\rho \otimes \tau^e}\alpha, \qquad E^e (f |_\rho \alpha) = (E^e f)|_{\rho \otimes \sigma^e}\alpha.
	\]
\end{enumerate}
\end{proposition}
\subsection{The relationship between elements of \texorpdfstring{$\U(\g_\C)$}{} and differential operators}
Let $(\rho, V)$ be a finite-dimensional representation of $\GL_n(\C)$ and let $f \in C^\infty(\H_n, V)$. Define $f^\rho \in C^\infty(\Sp_{2n}(\R), V)$ by
\begin{equation}\label{frho-denf}
 f^\rho(g) = (f |_\rho g)(iI_n).
\end{equation}
Note that if $f$ is a modular form with respect to a discrete subgroup $\Gamma$ of $\Sp_{2n}(\Q)$, i.e., if $f|_\rho \gamma = f$ for all $\gamma \in \Gamma$, then $f^\rho$ is left $\Gamma$-invariant. Recall the maps $\iota_\pm$ defined in (\ref{iota-defn}) which map $u \in T$ to $\iota_\pm(u) \in \p_\pm$. The elements of $\p_\pm$ act on functions on $\Sp_{2n}(\R)$ in the usual way. Given any collection $\{\varepsilon_1, \ldots, \varepsilon_e\}$ of symbols $\pm$, we get a map from $\Sp_{2n}(\R)$ to ${\rm Ml}_e(T, V)$ as follows
$$
 T^e \ni (u_1, \ldots, u_e) \mapsto \big(\iota_{\varepsilon_1}(u_1) \ldots \iota_{\varepsilon_e} (u_e) f^\rho\big)(g), \quad g \in \Sp_{2n}(\R).
$$
The following proposition gives the relation between the above action of $\p_\pm$ on $V$-valued functions on the group and the differential operators $D_\rho$ and $E$ acting on $V$-valued functions on the Siegel upper half space.
\begin{proposition}\label{alg-op-reln-prop}
Let $(\rho, V)$ be a finite dimensional representation of $\GL_n(\C)$. Let $f$ be in $C^\infty(\H_n, V)$ and let $f^\rho \in C^\infty(\Sp_{2n}(\R), V)$ be the corresponding function defined in (\ref{frho-denf}). Let $u_1, \ldots, u_e \in T$, and $g \in \Sp_{2n}(\R)$.
\begin{enumerate}
\item We have
\begin{align*}
	(\iota_{+}(u_1) \ldots \iota_{+}(u_e) f^{\rho})(g)  &=  ((D_{\rho}^{e}f)^{\rho \otimes \tau^{e}}(g)) (u_1, \ldots , u_e), \\
	(\iota_{-}(u_1) \ldots \iota_{-}(u_e) f^{\rho})(g)  &=  ((E^{e} f)^{\rho \otimes \sigma^{e}}(g)) (u_1, \ldots , u_e).
\end{align*}
\item Set $\mathcal{D}_\rho^+ = D_\rho$, $\mathcal{D}_\rho^- = E$, $\mu^+ = \tau$ and $\mu^- = \sigma$. Let $\varepsilon_1, \ldots, \varepsilon_e \in \{\pm\}$. Then
  \begin{align*}
	&(\iota_{\varepsilon_1}(u_1) \ldots \iota_{\varepsilon_e} (u_e) f^{\rho})(g)\\
	&\hspace{10ex}= (\mathcal{D}^{\varepsilon_1}_{\rho\otimes\mu^{\varepsilon_2} \otimes \cdots \otimes \mu^{\varepsilon_e}} \mathcal{D}^{\varepsilon_2}_{\rho\otimes\mu^{\varepsilon_3} \otimes \cdots \otimes \sigma^{\mu_e}} \cdots \mathcal{D}^{\varepsilon_e}_{\rho} f)^{\rho \otimes \mu^{\varepsilon_1} \otimes \cdots \otimes \mu^{\varepsilon_e}}(g)(u_1,\ldots,u_e).
  \end{align*}
\end{enumerate}
\end{proposition}
\begin{proof}
Part 1 follows from \cite[Proposition 7.3]{Sh90}, part 1 of Proposition \ref{DE-behav-prop}, and $E^{e+1} = EE^e$. For part 2, we proceed inductively in $e$. If $e=1$, it is already proven in part 1. Put
\[
 F = \mathcal{D}^{\varepsilon_2}_{\rho\otimes\mu^{\varepsilon_3} \otimes \cdots \otimes \mu^{\varepsilon_e}} \cdots \mathcal{D}^{\varepsilon_e}_{\rho} f.
\]
By the induction hypothesis,
\[
 F^{\rho \otimes \mu^{\varepsilon_2} \otimes \cdots \otimes \mu^{\varepsilon_e}}(g)(u_2,\ldots,u_e) = \iota_{\varepsilon_2}(u_2) \cdots \iota_{\varepsilon_e} (u_e) f^{\rho}(g),
\]
and by the $e=1$ case,
\begin{align}\label{case_r}
 \iota_{\varepsilon_1}(u_1) F^{\rho \otimes \mu^{\varepsilon_2} \otimes \cdots \otimes \mu^{\varepsilon_e}}(g) = (\mathcal{D}_{\rho \otimes \mu^{\varepsilon_2} \otimes \cdots \otimes \mu^{\varepsilon_e}}^{\varepsilon_1}F)^{\rho \otimes \mu^{\varepsilon_1} \otimes \cdots \otimes \mu^{\varepsilon_e}}(g)(u_1).
\end{align}
	Note that $F^{\rho \otimes \mu^{\varepsilon_2} \otimes \cdots \otimes \mu^{\varepsilon_e}}$ is a ${\rm Ml}_{e-1}(T, V)$-valued function on $\Sp_{2n}(\R)$.
	By substituting the tuple $(u_2,\ldots,u_e) \in T^{e-1}$ into (\ref{case_r}), the right hand side equals
	\begin{align*}
	&((\mathcal{D}_{\rho \otimes \mu^{\varepsilon_2} \otimes \cdots \otimes \mu^{\varepsilon_e}}^{\varepsilon_1}F)^{\rho \otimes \mu^{\varepsilon_1} \otimes \cdots \otimes \mu^{\varepsilon_e}}(g)(u_1))(u_2,\ldots,u_e)\\
	&\hspace{20ex}=(\mathcal{D}_{\rho \otimes \mu^{\varepsilon_2} \otimes \cdots \otimes \mu^{\varepsilon_e}}^{\varepsilon_1}F)^{\rho \otimes \mu^{\varepsilon_1} \otimes \cdots \otimes \mu^{\varepsilon_e}}(g)(u_1,u_2,\ldots,u_e).
	\end{align*}
	This equals
	\[
	(\mathcal{D}^{\varepsilon_1}_{\rho\otimes\mu^{\varepsilon_2} \otimes \cdots \otimes \mu^{\varepsilon_e}} \mathcal{D}^{\varepsilon_2}_{\rho\otimes\mu^{\varepsilon_3} \otimes \cdots \otimes \mu^{\varepsilon_e}} \cdots \mathcal{D}^{\varepsilon_e}_{\rho} f)^{\rho \otimes \mu^{\varepsilon_1} \otimes \cdots \otimes \mu^{\varepsilon_e}}(g)(u_1,\ldots,u_e).
	\]
	The left hand side of (\ref{case_r}) becomes
	\begin{align*}
	(\iota_{\varepsilon_1}(u_1) F^{\rho \otimes \mu^{\varepsilon_2} \otimes \cdots \otimes \mu^{\varepsilon_e}})(g)(u_2,\ldots,u_e)
	&= \iota_{\varepsilon_1}(u_1) (F^{\rho \otimes \mu^{\varepsilon_2} \otimes \cdots \otimes \mu^{\varepsilon_e}}(g)(u_2,\ldots,u_e)) \\
	&= \iota_{\varepsilon_1}(u_1)\iota_{\varepsilon_2}(u_2) \cdots \iota_{\varepsilon_e} (u_e) f^{\rho}(g).
	\end{align*}
This concludes the proof.
\end{proof}
\subsection{Scalar and vector-valued automorphic forms}\label{s:scalarvector}
Let $\Gamma$ be a congruence subgroup of $\Sp_{2n}(\Q)$ and let $V$ be a finite dimensional complex vector space. We let $\AA(\Gamma;V)$ denote the space of $V$-valued automorphic forms on $\Gamma\bs \Sp_{2n}(\R)$, i.e., the space of smooth $V$-valued functions on $\Sp_{2n}(\R)$ that are left $\Gamma$-invariant, $\mathcal{Z}$-finite, $K_\infty$-finite and slowly increasing. Let $\AA(\Gamma;V)^\circ \subset \AA(\Gamma;V)$ be the subspace of $V$-valued cusp forms. Let $\AA(\Gamma):= \AA(\Gamma; \C)$ and $\AA(\Gamma)^\circ:= \AA(\Gamma; \C)^\circ$ denote the usual spaces of (complex valued) automorphic forms and cusp forms on $\Gamma \bs \Sp_{2n}(\R)$.

Let $(\rho, V)$ be a rational, finite dimensional representation of $\GL_n(\C)$. We will abuse notation and use $\rho$ to denote its restriction\footnote{Recall that the restriction functor is an equivalence between the categories of  rational, irreducible, finite dimensional representations of $\GL_n(\C)$ and $U(n)$.} to  $U(n)$, as well as use $\rho$ to denote the corresponding representation of $K_\infty$ via the identification $K_\infty \overset{\sim}{\rightarrow} U(n)$ given by $\mat{A}{B}{-B}{A} \mapsto A+iB$, as defined earlier. Concretely, for $g\in K_\infty$, the element $J(g, iI_n)$ belongs to $U(n)$ and our identification means that $\rho(\bar{g})= \rho^{-1}(\iota(g)) = {\rho}(J(g, iI_n))$ for $g=\mat{A}{B}{-B}{A}\in K_\infty$, where the involution $\bar{g}$ on $K_\infty$ corresponds to complex conjugation on $U(n)$ and the anti-involution $\iota$ on $K_\infty$ corresponds to the transpose on $U(n$), i.e., $\bar{g} = \mat{A}{-B}{B}{A}$, $\iota(g) = \mat{{}^tA}{{}^tB}{-{}^tB}{{}^tA}$. The derived map, also denoted by $\iota$, is an anti-involution of $\mathfrak{gl}_n(\C)$ and of $\mathfrak{k}_\C$, and it extends naturally to an anti-involution of their respective universal enveloping algebras.

Given $f \in C^\infty(\H_n, V)$, the function $f^\rho$ defined in \eqref{frho-denf} satisfies  \begin{equation}\label{e:ktranseq}f^\rho(gk)=\rho(\iota(k))f^\rho(g)\end{equation}  for all $g\in\Sp_{2n}(\R)$,  $k\in K_\infty$. We will abuse notation and also use $\rho$ to denote the derived representation of $\mathfrak{gl}_n(\C)\simeq \mathfrak{k}_\C$, i.e.,
for $X\in \mathfrak{gl}_n(\C)$, we denote  \begin{equation}\label{e:notationlieaction}
\rho(X)v := \frac{d}{dt}\big|_0\rho(\exp(tX))v.
\end{equation} This action extends to $\U(\mathfrak{gl}_n(\C)) \simeq \mathcal{U}(\mathfrak{k}_\C)$.  It can be checked  that if $f^\rho \in C^\infty(\Sp_{2n}(\R), V)$ satisfies \eqref{e:ktranseq} then it also satisfies \begin{equation}\label{e:liealgebraactionvec}
  (Xf^\rho)(g)=\rho(\iota(X))(f^\rho(g))
\end{equation} for $X\in\mathcal{U}(\mathfrak{k}_\C)$ and $g\in\Sp_{2n}(\R)$.

We fix once and for all a \emph{rational structure} on $\rho$, i.e., we fix a $\Q$-vector space $V_\Q$ such that $V_\Q \otimes \C = V$ and such that the representation $\rho$ restricts to a homomorphism from $\GL_n(\Q)$ to $\GL_\Q(V_\Q)$.

We define
\begin{equation}\label{Arhogammadefeq}
 \AA_\rho(\Gamma) = \{f \in \AA(\Gamma; V)\ |\ f(gk)=\rho(\iota(k))f(g) \text{ for all }k\in K_\infty\}, \quad \AA_\rho(\Gamma)^\circ = \AA_\rho(\Gamma) \cap \AA(\Gamma;V)^\circ.
\end{equation}
Recall here that $\rho(\iota(k)) = \rho^{-1}(J(k, i I_n))$.

We consider the following general result.
\begin{lemma}\label{l:contraabstract}
 Let $G$ be a group, $K$ a subgroup, $(\sigma,V)$ a finite-dimensional, irreducible representation of $K$, and let $f:G\to V$ be a function satisfying
 $$
  f(gk)=\sigma(k)^{-1}f(g)\qquad\text{for all }g\in G,\:k\in K.
 $$
 Let $(\hat\sigma,\hat V)$ be the contragredient of $(\sigma,V)$.
 \begin{enumerate}
  \item  The $K$-module spanned by the right $K$-translates of $f$ is isomorphic to $\dim(V)$ copies of~$\hat\sigma$.
  \item Let $L$ be any fixed non-zero linear map from $V$ to $\C$, and define the $\C$-valued function $f_0$ on $G$ by  $f_0:= L \circ f$. Then the $K$-module spanned by the right $K$-translates of $f_0$ is isomorphic to~$\hat\sigma$.
 \end{enumerate}
\end{lemma}
\begin{proof}
For part 1, let $W$ be the space spanned by all right $K$-translates of $f$. For $L\in\hat V$, we define a map $\alpha_L:W\to\hat V$ by
$$
 \alpha_L\Big(\sum_ic_if(\cdot k_i)\Big)=\sum_ic_i\hat \sigma(k_i)L\qquad (c_i\in\C,\:k_i\in K).
$$
It is easy to verify that $\alpha_L$ is well-defined, linear, non-zero if $L\neq0$, and satisfies
$$
 \alpha_L(k.w)=\hat\sigma(k)\alpha_L(w)\qquad\text{for all }w\in W,\:k\in K,
$$
where $k.w$ means right translation. Hence, if $L\neq0$, then $\alpha_L$ is a surjection $W\to\hat V$ that intertwines the right translation action on $W$ with $\hat\sigma$. We claim that the linear map
\begin{equation}\label{l:contraabstracteq1}
 \hat V\longrightarrow\Hom_K(W,\hat V),\qquad L\longmapsto\alpha_L,
\end{equation}
is an isomorphism. Clearly, it is injective. To see the surjectivity, let $\beta:W\to\hat V$ be an intertwining operator. Set $L:=\beta(f)$. Then it is easy to see that $\beta=\alpha_L$.

We proved that $W$ contains $n:=\dim\hat V=\dim V$ copies of $\hat\sigma$. To see that $W$ contains no other representations, let $L_1,\ldots,L_n$ be a basis of $\hat V$, and consider the map
\begin{align}\label{l:contraabstracteq2}
 W&\longrightarrow\hat V\times\ldots\times\hat V\qquad\text{($n$ copies)},\\
 w&\longmapsto(\alpha_{L_1}(w),\ldots,\alpha_{L_n}(w)),\nonumber
\end{align}
It is straightforward to verify that this map is injective.
Considering dimensions, we see that the map \eqref{l:contraabstracteq2} is in fact an isomorphism, and that $W$ consists of $n=\dim V$ copies of $\hat V$ as a $K$-module.

For part 2, we change notation and let $W$ be the space spanned by the right $K$-translates of $f_0$. Then, with similar arguments as above, we see that the map
\begin{align*}
 \alpha:\:W&\longrightarrow\hat V,\\
 \sum_ic_if_0(\cdot k_i)&\longmapsto\sum_ic_i\hat\sigma(k_i)L\qquad (c_i\in\C,\:k_i\in K)
\end{align*}
is well-defined, injective, and commutes with the $K$-action, so that $W\cong\hat V$ as $K$-modules.
\end{proof}

\begin{lemma}\label{l:contramodule}
 Assume that the representation $(\rho,V)$ of $\GL_n(\C)$ is irreducible.
 \begin{enumerate}
  \item For each $f \in \AA_\rho(\Gamma)$, the $K_\infty$-module generated by the $K_\infty$-translates of $f$ is isomorphic to $\dim(V)\rho$.
  \item Let $L$ be any fixed non-zero linear map from $V$ to $\C$, and define the $\C$-valued function $f_0$ on $\Sp_{2n}(\R)$ by  $f_0:= L \circ f$. Then the $K_\infty$-module generated by the $K_\infty$-translates of $f_0$ is isomorphic to $\rho$.
 \end{enumerate}
\end{lemma}
\begin{proof}
Let $\sigma$ be the representation of $K_\infty$ (on the same representation space $V$ as of $\rho$) defined by $\sigma(k):= \rho(\iota(k)^{-1})$.  Then it is well-known that the contragredient representation $\hat{\sigma}$ is isomorphic to  $\rho$. Recall that the function $f$ satisfies $f(gk)=\sigma(k)^{-1}f(g)$ for all $g\in G,\:k\in K$. Hence our assertions follow from Lemma \ref{l:contraabstract}
\end{proof}

Recall that the set of (isomorphism classes of) irreducible rational representations $\rho$ of $\GL_n(\C)$  is parameterized via their highest weights by integers $k_1 \ge k_2 \ge \ldots \ge k_n$, i.e., elements of $\Lambda^+$.  More precisely, we say that such a $\rho$ has highest weight $(k_1, k_2, \ldots, k_n)$ if there exists $v_\rho \in V$  that, under the action of the Lie algebra $\mathfrak{k}_\C$, satisfies  $\rho(B_{i,j}) v_\rho = 0$ for $i<j$ and $\rho(B_{i,i}) v_\rho = k_i v_\rho$.  As recalled already, such a $v_\rho \in V$ is known as a highest weight vector in $V_\rho$ and is unique up to multiples. For $\k = (k_1, k_2, \ldots, k_n) \in \Lambda^+$, we  use $\rho_\k$ to denote the irreducible rational representation of $\GL_n(\C)$ (as well as of $U(n)$ and of $K_\infty$) with highest weight $\k$. We remark that the highest weight $(k,k,\ldots, k)$ corresponds to the character $\det^k$ of $\GL_n(\C)$ and the character $\det(J(g, iI_n))^{-k}$ of $K_\infty$.

It is well-known that an irreducible rational  representation $\rho$ of $\GL_n(\C)$ with highest weight  $(k_1, k_2, \ldots, k_n)$ with $k_n \ge 0$ is a polynomial representation of \emph{homogeneous degree} $k(\rho) = k_1 +k_2 +\ldots + k_n$. More generally, given any finite-dimensional (not necessarily irreducible) representation $\rho$ of $\GL_n(\C)$, we say that $\rho$ is polynomial of homogeneous degree $k(\rho)$  if $\rho$ is a direct sum of irreducible polynomial representations and
\begin{equation}\label{e:homdef}
 \rho(tg) = t^{k(\rho)} \rho(g) \text{ for all } t \in \C^\times,  \ g \in \GL_n(\C).
\end{equation}
Given an irreducible rational representation $(\rho, V)$, we fix a highest weight vector $v_\rho$ that is rational with respect to the given rational structure on $\rho$. Let $\langle , \rangle$ be the unique $U(n)$-invariant inner product on $V$ normalized such that $\langle v_\rho, v_\rho\rangle = 1$. Let $L_\rho: V \rightarrow \C$ be the orthogonal projection operator from $V$ to $\C v_\rho$ followed by the isomorphism $\C v_\rho \cong \C$. Equivalently, we may define $L_\rho(w) := \langle w, v_\rho\rangle$.

Given $\rho$ as above, let $\AA(\Gamma; \rho) \subset \AA(\Gamma)$ denote  the set of all $\C$-valued automorphic forms $f$ in $\AA(\Gamma)$ with the following properties: a) Either $f$ equals 0, or the $K_\infty$-module generated by the $K_\infty$-translates of $f$ is isomorphic to $\rho$, b) $f$ is a highest weight vector in the above module. Using the fact that the highest weight vector in $\rho$ is unique up to multiples, we see that  $\AA(\Gamma; \rho)$ is a \emph{vector-space}. One may equivalently define $\AA(\Gamma; \rho)$ as the space of all the automorphic forms $f$ satisfying $B_{i,j} f =0$ for $1\le i<j\le n$ and $B_{i,i} f =k_i f$ where $(k_1, \ldots, k_n)$ is the highest weight of $\rho$. Let $\AA(\Gamma; \rho)^\circ = \AA(\Gamma; \rho) \cap \AA(\Gamma)^\circ.$ We have the following lemma.
\begin{lemma}\label{l:Lrho}Let $(\rho, V)$ be an irreducible, rational, finite dimensional representation of $\GL_n(\C)$. The map $f\mapsto L_\rho \circ f$ gives an isomorphism of vector spaces $\AA_\rho(\Gamma) \rightarrow \AA(\Gamma; \rho)$ and $\AA_\rho(\Gamma)^\circ \rightarrow \AA(\Gamma; \rho)^\circ$.
\end{lemma}
\begin{proof}
Let $V$ be the space of $\rho$ and let $f \in \AA_\rho(\Gamma)$. By Lemma \ref{l:contramodule}, the $K_\infty$-translates of the function $L_\rho \circ f$ generate a module $W$ isomorphic to $\rho$.  A calculation shows that
\begin{equation}\label{l:Lrhoeq1}
 (X(L_\rho\circ f))(g)=\langle f(g),\rho(X)v_\rho\rangle
\end{equation}
for $g\in\Sp_{2n}(\R)$ and $X\in\mathfrak{k}_\C$. It follows that $L_\rho\circ f$ is a highest weight vector in $W$. Therefore, $L_\rho \circ f \in \AA(\Gamma; \rho).$  Next, we show that the map $f\mapsto L_\rho \circ f$  is injective. If not, there exists some non-zero $f \in \AA_\rho(\Gamma)$ such that $L_\rho(f(g)) = 0$ for all $g$; this contradicts the fact that the vectors $f(g)$, as $G$ runs through $\Sp_{2n}(\R)$, span all of $V$.

To complete the proof we need to show that the map $f\mapsto L_\rho \circ f$ is surjective. Suppose that $h \in \AA(\Gamma; \rho).$ Let $V_h$ be the space generated by the right-translates $R(k)h$ with $k \in K_\infty$. By assumption, there exists an isomorphism $\tau: V_h \rightarrow V$ as $K_\infty$-modules. Consider the element $\tilde{h} \in \AA(\Gamma; V)$ given by
$$
 \tilde{h}(g) = \int\limits_{K_\infty}h(g \iota(k^{-1})) \tau(R(k)h)\,dk,
$$
where $dk$ is a Haar measure on $K_\infty$. A calculation shows that $\tilde{h} \in \AA_\rho(\Gamma)$. Using the Schur orthogonality relations, one can further show that $L_\rho \circ \tilde{h}(g) = \langle \tilde{h}(g), v_\rho \rangle$ is a constant multiple of $h(g)$. By renormalizing  if needed, we obtain that  $L_\rho \circ \tilde{h} = h$. This completes the proof that $f\mapsto L_\rho \circ f$ gives an isomorphism $\AA_\rho(\Gamma) \rightarrow \AA(\Gamma; \rho)$. Finally, it is clear that  $f$ is cuspidal if and only if  $L_\rho \circ f$ is.
\end{proof}

\section{Nearly holomorphic Siegel modular forms}\label{s:nearholofull}
In this section we define nearly holomorphic Siegel modular forms and reframe them in the representation-theoretic language, construct some key  differential operators linking holomorphic and nearly holomorphic forms, and prove the arithmetic properties of these operators that will be crucial for our main theorem on $L$-values.

\subsection{Definitions and basic properties}
Let $(\rho, V)$ be a finite dimensional representation of $\GL_n(\C)$.
 Write $z \in \H_n$ as $z = x + iy$. For an integer $e \geq 0$, define $N_\rho^e(\H_n)$ to be the space of all functions $f: \H_n \rightarrow V$ of the form $f = \sum_i f_i v_i$ where $v_i$ ranges over some fixed basis\footnote{This definition does not depend on the choice of basis.} of $V$  and each $f_i:\H_n \rightarrow \C$ is a  polynomial of degree $\leq e$ in the entries of $y^{-1}$ with holomorphic functions $\H_n\to\C$ as coefficients. Note that the definition does not involve the representation $\rho$ but merely the representation space $V$. The space
$$N_\rho(\H_n) = \bigcup_{e \geq 0} N_\rho^e(\H_n)$$
is called the space of $V$-valued nearly holomorphic functions. By \cite[Lemma 13.3 (3) and equation (13.10)]{shibook2}
 we see that for $f \in C^\infty(\H_n, V)$, we have
\begin{equation}\label{nh-characterization}
f \in N_\rho^e(\H_n) \text{ if and only if } E^{e+1}f = 0,
\end{equation}
where $E$ is defined as in \eqref{e:defCE}. Now, by Proposition \ref{alg-op-reln-prop}, part 1, we conclude that
\begin{equation}\label{nh-implies-p-finite}
f \in N_\rho(\H_n) \text{ if and only if } \mathcal{U}(\p_-)f^\rho \text{ is finite-dimensional},
\end{equation}
\begin{equation}\label{h-implies-p-ann}
f \in N^0_\rho(\H_n) \text{ if and only if } f^\rho \text{ is annihilated by } \p_-.
\end{equation}
We say that a $(\g, K_\infty)$-module $V'$ is locally $(\p_-)$-finite, if $\mathcal{U}(\p_-)v$ is finite-dimensional for all $v \in V'$. It follows from the above and \eqref{pkcommeq} that if $f  \in N_\rho(\H_n)$, then $\U(\g_\C)f^\rho$ is a locally $(\p_-)$-finite $(\g, K_\infty)$-module.

For a congruence subgroup $\Gamma$ of $\Sp_{2n}(\Q)$, let $N_\rho(\Gamma)$ be the space of nearly holomorphic modular forms of weight $\rho$ with respect to $\Gamma$. Precisely, $N_\rho(\Gamma)$ consists of the space of functions $f \in  N_\rho(\H_n)$ such that $f |_\rho \gamma = f$ for all $\gamma \in \Gamma$; if $n=1$, we also require that the Fourier expansion of $f|_\rho \gamma$ is supported on the non-negative rationals for all $\gamma \in \Sp_{2n}(\Z)$. 
Let $N_\rho^e(\Gamma) = N_\rho(\Gamma) \cap N_\rho^e(\H_n)$. Any $f \in N_\rho^e(\Gamma)$ has a Fourier expansion of the form
\begin{equation}\label{e:fouriernearholo}
	f(z) = \sum\limits_{h \in M_n^{\rm sym}(\Q)}  c_h((\pi y)^{-1}) \exp(2 \pi  i\, {\rm tr}(hz))
\end{equation}
where  $c_h\in \bigcup_{0\le p \le e}S_p(T,V)$.
We use $N_\rho(\Gamma)^\circ$ to denote the space of cusp forms in $N_\rho(\Gamma)$, which consists of the forms $f$ for which the Fourier expansion of $f|_\rho \gamma$ is supported on positive definite matrices for all $\gamma \in \Sp_{2n}(\Z)$.

For $\k = (k_1, k_2, \ldots, k_n)$, we denote $$N_\k(\Gamma)= N_{\rho_\k}(\Gamma), \quad N^e_\k(\Gamma)= N^e_{\rho_\k}(\Gamma), \quad N_\k(\Gamma)^\circ= N_{\rho_\k}(\Gamma)^\circ, \quad N^e_\k(\Gamma)^\circ= N^e_{\rho_\k}(\Gamma)^\circ.$$
Furthermore, for a non-negative integer $k$, we denote $N_k(\Gamma)= N_{\det^k}(\Gamma)$, $N_k(\Gamma)^\circ= N_{\det^k}(\Gamma)^\circ.$ Note that $N_k(\Gamma) = N_{k,k, \ldots, k}(\Gamma)$.
Given forms $f_1$, $f_2$ in $N_k(\Gamma)^\circ$, we define the Petersson inner product $\langle f_1, f_2 \rangle$ by
\begin{equation}\label{e:defpetersson}
 \langle f_1, f_2 \rangle = \vl(\Gamma \bs \H_n)^{-1}\int\limits_{\Gamma \bs \H_n} \det(y)^k f_1(Z) \overline{f_2(z)}\,dz.
\end{equation}
Above, $dz$ is any $\Sp_{2n}(\R)$-invariant measure on $\H_n$ (the definition of the inner product does not depend on the choice of $dz$).

Finally we denote the set of holomorphic cusp forms as follows:
$$S_{\rho}(\Gamma) = N^0_\rho(\Gamma)^\circ, \quad S_\k(\Gamma) = N^0_\k(\Gamma)^\circ.$$
\subsection{Nearly holomorphic modular forms and \texorpdfstring{$(\p_-)$}{}-finite automorphic forms}\label{s:ident}
We let $\AA(\Gamma; V)_{\p_-\text{-fin}}$ denote the subspace of $\AA(\Gamma; V)$ consisting of all $f \in \AA(\Gamma; V)$ such that $\mathcal{U}(\p_-)f$  is finite-dimensional.
If $V=\C$ we denote this space by $\AA(\Gamma)_{\p_-\text{-fin}}$. Let $\AA(\Gamma; V)_{\p_-\text{-fin}}^\circ$ and $\AA(\Gamma)_{\p_-\text{-fin}}^\circ$ denote the corresponding spaces of cusp forms. It is easy to see that $\AA(\Gamma; V)_{\p_-\text{-fin}}$,  $\AA(\Gamma; V)_{\p_-\text{-fin}}^\circ$, $\AA(\Gamma)_{\p_-\text{-fin}}$ and $\AA(\Gamma)_{\p_-\text{-fin}}^\circ$ are all locally $(\p_-)$-finite $(\g, K_\infty)$-modules.

Define
\[
 \AA_\rho(\Gamma)_{\p_-\text{-fin}}= \AA(\Gamma; V)_{\p_-\text{-fin}} \cap \AA_\rho(\Gamma) ,\quad \AA_\rho(\Gamma)_{\pp}^\circ = \AA(\Gamma; V)_{\p_-\text{-fin}} \cap \AA_\rho(\Gamma)^\circ.
\]
\[\AA(\Gamma; \rho)_{\p_-\text{-fin}}= \AA(\Gamma)_{\p_-\text{-fin}} \cap \AA(\Gamma; \rho) ,\quad \AA(\Gamma; \rho)_{\pp}^\circ = \AA(\Gamma)_{\p_-\text{-fin}} \cap \AA(\Gamma; \rho)^\circ.
\]
Recall that $\AA_\rho(\Gamma)$ was defined in \eqref{Arhogammadefeq} and $\AA(\Gamma; \rho)$ was defined before Lemma \ref{l:Lrho}.
The following crucial proposition, which generalizes Proposition 4.5 of \cite{PSS14}, gives the relation between nearly holomorphic modular forms and $(\p_-)$-finite automorphic forms. For a function $f$ on $\H_n$, recall the associated function $f^\rho$ on $\Sp_{2n}(\R)$ defined in \eqref{frho-denf}.
\begin{proposition}\label{p:iso}
 \begin{enumerate}
  \item The map  $f \mapsto f^\rho$ gives isomorphisms of vector spaces  $N_{\rho}(\Gamma) \isomto \AA_\rho(\Gamma)_{\pp}$ and $N_{\rho}(\Gamma)^\circ \isomto \AA_\rho(\Gamma)_{\pp}^\circ$.
  \item Assume that $\rho$ is irreducible. The map $f \mapsto L_\rho \circ f^\rho$ gives isomorphisms of vector spaces  $N_{\rho}(\Gamma) \isomto  \AA(\Gamma; \rho)_{\pp}$ and $N_{\rho}(\Gamma)^\circ  \isomto \AA(\Gamma; \rho)_{\pp}^\circ$.
 \end{enumerate}
\end{proposition}
\begin{proof}
In view of Lemma \ref{l:Lrho} and the fact that $L_\rho$ preserves the $(\p_-)$-finiteness property, it suffices to prove the first part of the proposition.

Let us first show that for each $f$ in $N_{\rho}(\Gamma)$, the function  $f^\rho$ lies in $\AA_\rho(\Gamma)_{\pp}$. Clearly, $f^\rho$ is left $\Gamma$-invariant and $K_\infty$-finite. Let $\pi = \U(\g_\C) f^\rho$. Since $f$ is a nearly holomorphic modular form, $\pi$ is a $(\g, K_\infty)$-module, and is locally $\p_-$-finite; see (\ref{nh-implies-p-finite}). The action of the Cartan subalgebra $\h$ is semisimple, hence $\pi$ is a weight module. This implies that $\pi$ lies in Category $\mathcal{O}$ in the sense of \cite{Hum08}. 
By \cite[Thm 1.1 (e)]{Hum08} we have that $f^\rho$ is  $\mathcal Z$-finite. Finally, $f^\rho$ has the  moderate growth condition from Theorem 1.1 of \cite{PiSS17}. This completes the proof that $f^\rho$ lies in $\AA_\rho(\Gamma)_{\pp}$.

To show that the map  $f \mapsto f^\rho$ is an isomorphism, we construct an inverse map.  Let $f' \in \AA_\rho(\Gamma)_{\pp}$, $z \in \H_n$ and $g \in \Sp_{2n}(\R)$ such that $g( iI_n) = z$. Then define $f(z) := \rho(J(g, iI_n))f'(g)$. Then $f \in N_{\rho}(\Gamma)$ by the left $\Gamma$-invariance of $f'$ and (\ref{nh-implies-p-finite}) and it can be easily checked that the map $f' \mapsto f$ defined above is the inverse of the map $f \mapsto f^\rho$.

Finally, the argument of Proposition 4.5 of \cite{PSS14} shows that $f^\rho$ is cuspidal if and only if $f$ is.
\end{proof}

For the rest of this subsection, we assume that $\rho$ is irreducible. The above result implies that for any $f \in N_{\rho}(\Gamma)^\circ$ the function $L_\rho \circ f^\rho$ is a cuspidal automorphic form on $\Sp_{2n}(\R)$ and the $(\g, K_\infty)$-module generated by $L_\rho \circ f^\rho$  decomposes into a finite direct sum of irreducible, admissible, unitary $(\g, K_\infty)$-modules. We now make this observation more precise.

For each $\k \in \Lambda^+$, let $F_\k$ be any model for $\rho_\k$, and consider $F_\k$ as a module for $\mathfrak{k}_\C + \p_{-}$ by letting $\p_{-}$ act trivially. Let $N(\k) :=  \U(\g_\C) \otimes_{\U(\mathfrak{k}_\C + \p_{-})} F_\k$. Then $N(\k)$ is a locally $\p_{-}$-finite $(\g, K_\infty)$ module, and by the general theory of category $\O^\p$ (see Section 9.4 of \cite{Hum08}), it admits a unique irreducible quotient, which we denote by $L(\k)$. The $(\g, K_\infty)$ module $L(\k)$ is locally $\p_-\text{-finite}$ and contains the $K_\infty$-type $\rho_\k$ with multiplicity one.

From the theory developed in Chapter 9 of \cite{Hum08}, \emph{any}  irreducible, locally $\p_-\text{-finite}$ $(\g, K_\infty)$-module is isomorphic to  $L(\k)$ for some $\k \in \Lambda^+$. More precisely, given an irreducible, locally $\p_-\text{-finite}$ $(\g, K_\infty)$-module  $V$, there exists a $\k \in \Lambda^+$ and a $v \in V$ such that $L(\k) \cong \U(\g_\C) v$ and the following properties hold: a) $v$ is a highest weight vector of weight $\k$ generating the $K_\infty$-type $\rho_\k$, b) $v$ is annihilated by  $\p_-$. These two properties identify $v \in V$ uniquely up to multiples, and ensure that $V\simeq L(\k)$.

For any $\k \in \Lambda^+$, we define $\AA(\Gamma; L(\k))^\circ$ to be the subspace of $\AA(\Gamma)^\circ$ spanned by the forms $f$ such that $\U(\g_\C) f$ is isomorphic to a sum of copies of $L(\k)$ as a $(\g, K_\infty)$-module. Since $L(\k)$ is locally $\p_-\text{-finite}$, it follows that
$$
 \AA(\Gamma; L(\k))^\circ\subseteq \AA(\Gamma)^\circ_\pp.
$$
Hence $\AA(\Gamma; L(\k))^\circ$ is the $L(\k)$-isotypical component of $\AA(\Gamma)^\circ_\pp$.

 \begin{proposition}\label{AA0nfindecomp2prop}\label{l:multipli}
 As $(\mathfrak{g},K_\infty)$-modules, we have
 $$
  \AA(\Gamma)^\circ_\pp=\bigoplus_{\k \in \Lambda^+}\,\AA(\Gamma; L(\k))^\circ \text{ where } \AA(\Gamma; L(\k))^\circ \cong \dim(S_\k(\Gamma)) \,L(\k).
 $$
 The highest weight vectors of weight $\k$ in $\AA(\Gamma; L(\k))^\circ$ correspond to elements of $S_{\k}(\Gamma)$ via the map from Proposition \ref{p:iso}. The  direct sum decomposition above is orthogonal with respect to the Petersson inner product.
\end{proposition}
\begin{proof}
The existence of the Petersson inner product implies that the $(\g, K_\infty)$-module $\AA(\Gamma)_{\p_-\text{-fin}}^\circ$ decomposes as a direct sum of irreducible, locally $\p_-\text{-finite}$ $(\g, K_\infty)$-modules. As noted earlier, any such module is isomorphic to $L(\k)$ for some $\k \in \Lambda^+$.

Next we show that the $(\g, K_\infty)$-module $L(\k)$ occurs in  $\AA(\Gamma)^\circ_{\pp}$ with multiplicity equal to $\dim(S_\k(\Gamma))$. For this, let $W'$ be the subspace of $\AA(\Gamma; L(\k))^\circ$  spanned by all the highest weight vectors of weight $\k$. In each copy of $L(\k)$, the highest weight vector of weight $\k$ is unique up to multiples, and $W'$ is spanned by these highest weight vectors. Therefore the dimension of $W'$ equals the multiplicity of  $L(\k)$  in  $\AA(\Gamma)^\circ_{\pp}$. It suffices to then show that the map $S_\k(\Gamma) \rightarrow W'$ given by $f \mapsto L_{\rho_\k}\circ f^{\rho_\k}$ is an isomorphism. Note here that $W'$ is the subspace of $\AA(\Gamma; \rho_\k)^\circ_{\pp}$ that is annihilated by $\p_{-}$. Now the required result follows from \eqref{h-implies-p-ann} and Proposition \ref{p:iso}.
\end{proof}

As is well-known, for $\k = (k_1, k_2, \ldots, k_n) \in \Lambda^+$, we have $S_\k(\Gamma) = \{0\}$ if  $k_n\le0$ (see \cite{Freitag79}). So a consequence of the above results is that $L(\k)$ cannot occur in  $\AA(\Gamma)^\circ_{\pp}$ when $k_n \le 0$.

\begin{lemma}\label{L:llambdaktypes}
 Let $\k', \k \in \Lambda^+$ with $\k'=(k_1', k_2', \ldots, k_n')$ and $\k=(k_1, k_2, \ldots, k_n)$. Suppose that $L(\k')$ contains the $K_\infty$-type $\rho_\k$. Then $k_i \ge k_i'$ for $1\le i \le n$.
\end{lemma}
\begin{proof} It suffices to prove the analogous statement for the representation $N(\k')$. Let $B_{i,j}$ be the elements defined in \eqref{BijEijdefeq}, and set $\mathfrak{k}_+=\langle B_{i,j}:1\le i<j\le n\rangle$  (the space spanned by the root vectors for the positive compact roots), $\mathfrak{k}_-=\langle B_{i,j}:1\le j<i\le n\rangle$. Then $\mathfrak{k}_\C=\mathfrak{h}_\C+\mathfrak{k}_++\mathfrak{k}_-$. Let $v_0$ be a highest weight vector in $F_{\k'}$, which we recall is a model for $\rho_{\k'}$. Then, as vector spaces,
$$
 N(\k') =  \U(\g_\C) \otimes_{\U(\mathfrak{k}_\C + \p_{-})} F_{\k'}
 =  \U(\p_+)F_{\k'}
 =  \U(\p_+)\U(\mathfrak{k}_-)v_0
 =  \U(\mathfrak{k}_-)\U(\p_+)v_0.
$$
The last equality follows from $[\mathfrak{k}_\C,\p_+]\subset\p_+$. It is clear that all the highest weight vectors of the $K_\infty$-types occurring in $N(\k')$ must be contained in $\U(\p_+)v_0$. Hence the weight of such a highest weight vector is of the form $(k_1,k_2,\ldots,k_n)$ with $k_i \ge k_i'$ for $1\le i \le n$.
\end{proof}

We can now prove the following fact.
\begin{proposition}\label{l:fact2}Let $\rho$ be a rational, irreducible, finite dimensional representation of $\GL_n(\C)$, and let $\Gamma$ be a congruence subgroup of $\Sp_{2n}(\Q)$. Then the space $N_{\rho}(\Gamma)^\circ \simeq \AA(\Gamma; \rho)_{\pp}^\circ$ is finite dimensional.
\end{proposition}
\begin{remark}A well-known result of Shimura \cite[Lemma 14.3]{shibook2} implies that $N^e_\rho(\Gamma)^\circ$ is finite-dimensional for each non-negative integer $e$. However, Proposition \ref{l:fact2} goes much further and asserts that the space $\bigcup_{e=1}^\infty N^e_\rho(\Gamma)^\circ$ is itself finite-dimensional, i.e., $N^e_\rho(\Gamma)^\circ = N^{e+1}_\rho(\Gamma)^\circ$ for all sufficiently large $e$.
\end{remark}
\begin{proof}
Let $\k = (k_1, k_2, \ldots, k_n)$ be such that $\rho= \rho_\k$.   Recall that each element of $\AA(\Gamma; \rho)_{\pp}^\circ$ generates a $K_\infty$-module isomorphic to $\rho$. Since a given irreducible $(\g, K_\infty)$-module contains the $K_\infty$-type $\rho$ with finite multiplicity, it suffices to show that the $(\g, K_\infty)$-module $U(\g_\C)\AA(\Gamma; \rho)_{\pp}^\circ$ decomposes into a direct sum of \emph{finitely many} irreducible $(\g, K_\infty)$-modules.

Using Proposition \ref{AA0nfindecomp2prop} and the remarks following it, let
\begin{equation}\label{e:decomposition}
 U(\g_\C)\AA(\Gamma; \rho)_{\pp}^\circ  = \oplus_{\lambda \in \Lambda^+} \, d_\lambda L(\lambda),
\end{equation}
where $d_\lambda \le \dim(S_\lambda(\Gamma))$ and $d_\lambda =0$ if $\lambda = (\lambda_1, \lambda_2, \ldots, \lambda_n)$ satisfies $\lambda_n \le 0$. We now claim that if $d_\lambda\ne 0$ then $\lambda_i \le k_i$ for all $i$. This will complete the proof, since there are only finitely many elements $(\lambda_1, \lambda_2, \ldots, \lambda_n) \in \Lambda^+$ satisfying $\lambda_n >0$ and $\lambda_i \le k_i$ for $1\le i \le n$.

To prove the aforementioned claim, we appeal to Lemma \ref{L:llambdaktypes}. Suppose that $d_\lambda \ne 0$. Then there exists an element $f$ in $\AA(\Gamma; \rho)_{\pp}^\circ$ whose  component $f_\lambda$ in \eqref{e:decomposition} along $d_\lambda L(\lambda)$ is non-zero. Clearly, $f_\lambda$ generates the $K_\infty$-type $\rho=\rho_\k$, and hence $\rho_\k$ occurs in $L(\lambda)$. Now, Lemma \ref{L:llambdaktypes} implies that $\lambda_i \le k_i$ for $1\le i \le n$.
\end{proof}

Next recall the space $\AA(\Gamma; L(\k))^\circ$ defined before Proposition \ref{AA0nfindecomp2prop}, and define
$$
 \AA(\Gamma; \rho,  L(\k))^\circ =  \AA(\Gamma; \rho)^\circ \cap \AA(\Gamma; L(\k))^\circ.
$$
If $\rho = \rho_{\lambda}$ then by the results above we have
\begin{equation}\label{kilambdaieq}
 \AA(\Gamma; \rho,  L(\k))^\circ  \neq \{0\} \quad \Rightarrow  \quad 1 \le k_i \le \lambda_i \text{ for } 1\le i \le n.
\end{equation}

Let $N_{\rho}(\Gamma; L(\k))^\circ \subset N_{\rho}(\Gamma)^\circ $ and $\AA_\rho(\Gamma;  L(\k))^\circ \subset \AA_\rho(\Gamma)_{\pp}^\circ $ be the isomorphic images of $\AA(\Gamma; \rho,  L(\k))^\circ$ under the isomorphisms $N_{\rho}(\Gamma)^\circ \isomto \AA_\rho(\Gamma)_{\pp}^\circ \isomto \AA(\Gamma; \rho)_{\pp}^\circ$ given by Proposition \ref{p:iso}. Then we have orthogonal (with respect to the Petersson inner product) direct sum decompositions of finite-dimensional vector spaces
\begin{equation}\label{e:decom}
 \AA(\Gamma; \rho)_{\pp}^\circ = \bigoplus_{\k \in \Lambda^{++}}  \AA(\Gamma; \rho,  L(\k))^\circ, \quad N_\rho(\Gamma)^\circ = \bigoplus_{\k \in \Lambda^{++}} N_\rho(\Gamma; L(\k))^\circ,
\end{equation}
where we recall that $\Lambda^{++}$ consists of those $\k=(k_1, \ldots, k_n) \in \Lambda^+$ with $k_n\geq1$. Above, if $\lambda = (\lambda_1, \ldots, \lambda_n)$ is the highest weight of $\rho$, then the sum in \eqref{e:decom} can be taken over the (finitely many) $\k=(k_1, \ldots, k_n) \in \Lambda^{++}$ such that $k_i \le \lambda_i$ for all $i$, since the summands are zero otherwise.
\begin{remark}The decomposition \eqref{e:decom} together with Proposition \ref{AA0nfindecomp2prop} may be  viewed as a \emph{structure theorem} for the space of nearly holomorphic cusp forms from the representation-theoretic point of view.
\end{remark}
\subsection{Action of \texorpdfstring{$\Aut(\C)$}{}}\label{s:rationalnice}
Recall that $T$ denotes the space of symmetric $n\times n$ complex matrices. Let $V$ be a vector space with a rational structure $V_\Q$, i.e., $V_\Q \subset V$ is a vector-space over $\Q$ such that $V_\Q \otimes \C = V$. Then $V_\Q$ gives rise to a rational structure on $\mathrm{Ml}_e(T,V)$, as follows. Let $\{v_1,\ldots,v_d\}$ be a basis of $V_\Q$. For $i_1,\ldots,i_e,j_1,\ldots,j_e\in\{1,\ldots,n\}$, define the function $f_{(i_1, j_1), \ldots, (i_e, j_e)}\in\mathrm{Ml}_e(T,\C)$ by
\begin{equation}\label{basicfsdefeq}
 f_{(i_1, j_1), \ldots, (i_e, j_e)}(y^{(1)},  \ldots, y^{(e)}) := y^{(1)}_{i_1, j_1} \cdots   y^{(e)}_{i_e, j_e}.
\end{equation}
Then the elements $f_{(i_1, j_1), \ldots, (i_e, j_e)}v_i$, where $i_1,\ldots,i_e,j_1,\ldots,j_e$ run through $\{1,\ldots,n\}$ and $i$ runs through $\{1,\ldots,d\}$, define a rational structure on $\mathrm{Ml}_e(T,V)$, which is independent of the choice of basis of $V_\Q$. We also obtain a rational structure on $S_e(T,V)$, which we recall can be identified with the symmetric elements of $\mathrm{Ml}_e(T,V)$.

Let $t_{i,j}$, $1 \le i \le j \le n$ be indeterminates. For $i,j$ as above, define $t_{j,i} = t_{i,j}$ and let $\mathcal{T}$ denote the symmetric $n \times n$  matrix of indeterminates whose $(i,j)$ entry equals $t_{i,j}$. For a field $F$, let $F[\mathcal{T}]$ denote the algebra of polynomials in the indeterminates $t_{i,j}$; clearly $F[\mathcal{T}]$ can be identified with the polynomial ring over $F$ in $\frac{n^2+n}{2}$ variables. We let $F(\mathcal{T})$ denote the field of fractions of $F[\mathcal{T}]$. We let $\mathcal{T}^{-1}$ denote the (formal) inverse of $\T$ according to Cramer's rule, i.e., $\mathcal{T}^{-1} = \frac{1}{\det \mathcal{T}} \mathrm{adj}(\mathcal{T})$ where $\mathrm{adj}$ denote the adjugate. We define the algebra $F[\mathcal{T}^\pm] \subset F(\mathcal{T})$ as follows.
\begin{equation}\label{nice-fns-defn}
F[\mathcal{T}^\pm] := \{\text{polynomials over $F$ in the entries of } \mathcal{T} \text{ and } \mathcal{T}^{-1}\}.
\end{equation}
It is an easy exercise that $F[\mathcal{T}^\pm]$ equals the extension of the polynomial ring $F[\mathcal{T}]$ by $\frac{1}{\det \mathcal{T}}$. Equivalently, $F[\mathcal{T}^\pm]$ is the localization of $F[\mathcal{T}]$  in the multiplicative set consisting of the non-negative powers of $\det(\mathcal{T})$. Given $c \in F[\mathcal{T}^\pm]$ we obtain a well-defined  \emph{function} from the set of \emph{invertible} symmetric matrices over $F$ to $F$. As a special case of this which will be relevant for us, recall that a real, positive-definite matrix $y$ has a unique positive definite square-root $y^{1/2}$ and so for any $c \in \C[\mathcal{T}^\pm]$, the complex number $c(y^{1/2})$ makes sense.

We say that a non-zero $c \in \C[\mathcal{T}^\pm]$ is homogeneous of degree $m$ if $c(r\mathcal{T}) = r^m c(\mathcal{T})$. This makes the space $\C[\mathcal{T}^\pm]$ into a graded algebra, graded by degree $m \in \Z$.
We say that an element $c \in \C[\mathcal{T}^\pm]$ is  \emph{rational} if  $c \in \Q[\mathcal{T}^\pm]$. Let
\begin{equation}\label{vector-nice-fn-defn}
 V[\mathcal{T}^\pm] := \C[\mathcal{T}^\pm] \otimes_\C V.
\end{equation}
We define an action of $\Aut(\C)$ on the space $V[\mathcal{T}^\pm]$ as follows.  Given $c \in V[\mathcal{T}^\pm]$, write it as $c(y) = \sum_i c_i(y) v_i$ where $v_i$ ranges over a fixed rational basis of $V_\Q$ and $c_i \in \C[\mathcal{T}^\pm]$. For $\sigma \in \Aut(\C)$, we define $$ \tst {}^\sigma c(y) = \sum_i {}^\sigma c_i(y) v_i$$ where ${}^\sigma c_i$ is obtained by letting $\sigma$ act on the coefficients of $c_i$. This gives a well defined action of $\Aut(\C)$ on the space $V[\mathcal{T}^\pm]$  (that does not depend on the choice of the basis $\{v_i\}$ of $V_\Q$). An element $c \in V[\mathcal{T}^\pm]$  is defined to be rational if the components $c_i$ all belong to $\Q[\mathcal{T}^\pm]$. We let $V_\Q[\mathcal{T}^\pm]$ denote the $\Q$-vector space of rational elements in $V[\mathcal{T}^\pm]$ so that  \[V_\Q[\mathcal{T}^\pm] = \Q[\mathcal{T}^\pm] \otimes_\Q V_\Q.\] It is clear that an element $c$ in  $V[\mathcal{T}^\pm]$ belongs to $V_\Q[\mathcal{T}^\pm]$ if and only if ${}^\sigma c = c$ for all $\sigma \in \Aut(\C)$.

We let $N_S(\H_n, V) \subset C^\infty(\H_n, V)$ consist of all functions $f \in C^\infty(\H_n, V)$ with the property that there exists an integer $N$ (depending on $f$) such that $f$ has an absolutely and uniformly (on compact subsets) convergent expansion
\begin{equation}\label{e:fouriernice}
 f(z) = f(x + iy) = \sum_{h \in \frac{1}{N}M_n^{\sym}(\Z)}\!q_h((\pi y)^{1/2}) \exp(2 \pi i\, {\rm tr}(hz)),
\end{equation}
where each $q_h=q_{h, f}$ is an element of $V[\mathcal{T}^\pm]$. Observe that for any finite dimensional rational representation $\rho$ of $\GL_n(\C)$ on $V$ we have $N_\rho(\Gamma) \subseteq N_S(\H_n, V)$. Indeed, if $f \in N_\rho(\Gamma)$, the corresponding $q_h$ in \eqref{e:fouriernice} actually belong to the subalgebra $\C[\mathcal{T}^{-2}] \otimes_\C V$ of $V[\mathcal{T}^\pm]$ (cf.~\eqref{e:fouriernearholo}).

Given a function $f \in N_S(\H_n, V)$, we  define  ${}^\sigma\!f \in N_S(\H_n, V)$ by
\begin{equation}\label{e:fourierniceautc}
 {}^\sigma\!f(z) = \sum_h {}^\sigma\!q_h((\pi y)^{1/2}) \exp(2 \pi i\, {\rm tr}(hz)).
\end{equation}
A very special case of all this is the (well-known) action of $\Aut(\C)$ on $N_\rho(\Gamma)$. Suppose that $\rho$ is a representation of $\GL_n(\C)$ on $V$ that respects the rational structure on $V$ (meaning that $\rho$ restricts to a homomorphism from $\GL_n(\Q)$ to $\GL_\Q(V_\Q)$). The action of $\Aut(\C)$ on elements of $S_e(T,V)$ leads to an action of $\Aut(\C)$ on $N_\rho(\Gamma)$, via
\begin{equation}\label{e:autcnearholo}
 \rule{0ex}{2.5ex}^\sigma\!\bigg(\sum_h \!c_h((\pi y)^{-1}) \exp(2 \pi i\, {\rm tr}(hz))\bigg)
	 = \sum_h {}^\sigma\!c_h ((\pi y)^{-1}) \exp(2 \pi i\, {\rm tr}(hz)),\end{equation} where $c_h \in \bigcup_{e} S_e(T,V)$. This is a special case of the definition \eqref{e:fourierniceautc}.
From Theorem 14.12 (2) of \cite{shibook2}, we get the following result in this special setup.
\begin{equation}\label{e:nearholoshimuraautc}
 \text{For } f \in N^e_\rho(\Gamma) \text{ and } \sigma \in \Aut(\C), \text{ we have } {}^\sigma\!f \in N_\rho^e(\Gamma).
\end{equation}
In the special case that $\sigma$ equals the complex conjugation, we will denote ${}^\sigma\!f$ by $\bar{f}$. Note that $\bar{f}(z) = \overline{f(-\bar{z})}.$
We will need the following result, which is Theorem 14.12 (3) of \cite{shibook2}.
\begin{proposition}\label{shimura}
	Let $f \in N_\rho(\Gamma)$. For a positive integer $p$, let $D_\rho^p$ and $E^p$ be defined as in \eqref{e:DefDrho}. Then we have
   $$
    {}^\sigma\!((\pi  i)^{-p}D_\rho^p f) = (\pi i)^{-p}D_\rho^p ({}^\sigma\!f), \qquad {}^\sigma\!((\pi  i)^{p}E^p f) = (\pi i)^pE^p ({}^\sigma\!f).
   $$
\end{proposition}

We will now prove some general results on $\Aut(\C)$-equivariance of operators, culminating in Proposition~\ref{p:master} below, which will be crucial for the results of the next subsection. To begin with, we prove a lemma that  will allow us to move from $\mathrm{Ml}_b(T,V)$-valued functions to $V$-valued functions, and will also clarify why we cannot restrict ourselves to the space $N_\rho(\Gamma)$ and instead need to consider the larger space $N_S(\H_n, V)$.

\begin{lemma}\label{l:trivaction}
 Let $e$, $b$ be positive integers, and for $1 \le i \le b$, let $u_i \in M_n^\sym(\Q)$ and $\epsilon_i \in \Z$.
 \begin{enumerate}
  \item Let $T_{>0}$ denote the set of real, symmetric, positive-definite,  $n \times n$ matrices. Let $q \in \bigcup_{a \le e}S_a(T, \mathrm{Ml}_b(T,V))$. Then there exists $r_q \in V[\T^\pm]$ such that for all $y \in T_{>0}$, $$r_q( y^{1/2}) = q( y^{-1})(y^{\frac{\epsilon_1}{2}} u_1  y^{\frac{\epsilon_1}{2}}, \ldots,  y^{\frac{\epsilon_b}{2}} u_b y^{\frac{\epsilon_b}{2}})
  $$
  and for any $\sigma \in \Aut(\C)$ we have
  $$
   {}^\sigma \! (r_q) = r_{{}^\sigma \!q}.
  $$
  \item Let $\rho$ be a finite dimensional, rational representation of $\GL_n(\C)$ on $\mathrm{Ml}_b(T,V)$ that respects the rational structure on $\mathrm{Ml}_b(T,V)$.  For each $f\in N_\rho(\Gamma)$, define the function $\theta f \in C^\infty(\H_n, V)$ by
  \[
   (\theta f)(z) = (f(z))((\pi y)^{\frac{\epsilon_1}{2}} u_1  (\pi y)^{\frac{\epsilon_1}{2}}, \ldots,  (\pi y)^{\frac{\epsilon_b}{2}} u_b (\pi y)^{\frac{\epsilon_b}{2}}).
  \]
  Then $\theta f \in N_S(\H_n, V)$ and for all $\sigma \in \Aut(\C)$ we have $\theta({}^\sigma \!f) = {}^\sigma \!( \theta f)$.
 \end{enumerate}
\end{lemma}
\begin{proof}
We first prove part 1 of the lemma. Recall the definition of the functions $f_{(i_1, j_1), \ldots, (i_b, j_b)}$ in \eqref{basicfsdefeq}, and the rational structure of $\mathrm{Ml}_b(T,V)$ given by the elements $f_{(i_1, j_1), \ldots, (i_b, j_b)}v_i$.
From our definitions, we have that the function from $T_{>0}$ to $V$ given by
$$
 y \mapsto f_{(i_1, j_1), \ldots, (i_b, j_b)} ( y^{\frac{\epsilon_1}{2}} u_1  y^{\frac{\epsilon_1}{2}}, \ldots,  y^{\frac{\epsilon_b}{2}} u_b y^{\frac{\epsilon_b}{2}})v_i
$$
is the evaluation of a function in $V_\Q[\mathcal{T}^\pm]$ at $y^{1/2}$, i.e., there is $g_{(i_1, j_1), \ldots, (i_b, j_b);i} \in V_\Q[\mathcal{T}^\pm]$ such that
\[
 g_{(i_1, j_1), \ldots, (i_b, j_b);i}(y^{1/2}) =  f_{(i_1, j_1), \ldots, (i_b, j_b)} ( y^{\frac{\epsilon_1}{2}} u_1  y^{\frac{\epsilon_1}{2}}, \ldots,  y^{\frac{\epsilon_b}{2}} u_b y^{\frac{\epsilon_b}{2}})v_i.
\]
Now write $q = \sum c_{q;(i_1, j_1), \ldots (i_b, j_b);i} f_{(i_1, j_1), \ldots, (i_b, j_b)}v_i$ where $c_{q;(i_1, j_1), \ldots, (i_b, j_b);i} \in \C[\T]$ are of degree $\le e$ and the sum runs over all $i_1,\ldots,i_b,j_1,\ldots,j_b$ in $\{1,\ldots,n\}$ and $v_i$ runs through a fixed rational basis of $V_\Q$. By definition,
\begin{equation}\label{e:autclemma1}
 {}^\sigma c_{q;(i_1, j_1), \ldots, (i_b, j_b);i} = c_{{}^\sigma\!q;(i_1, j_1), \ldots, (i_b, j_b);i} \quad \text{and}
\end{equation}
\begin{equation}\label{e:autclemma2}
 q( y^{-1})(y^{\frac{\epsilon_1}{2}} u_1  y^{\frac{\epsilon_1}{2}}, \ldots,  y^{\frac{\epsilon_b}{2}} u_b y^{\frac{\epsilon_b}{2}}) = \sum c_{q;(i_1, j_1), \ldots, (i_b, j_b);i} (y^{-1}) g_{(i_1, j_1), \ldots, (i_b, j_b);i}(y^{1/2}).
\end{equation}
Now define $r_q(\T) = \sum c_{q;(i_1, j_1), \ldots, (i_b, j_b);i} (\T^{-2}) g_{(i_1, j_1), \ldots, (i_b, j_b);i}(\T).$
Then $r_q \in V[\T^\pm]$ satisfies the required properties by \eqref{e:autclemma1}, \eqref{e:autclemma2}.

Finally, part 2 is an immediate consequence of part 1.
\end{proof}

\begin{lemma}\label{l:rationalityrho}Let $f$ in $N_S(\H_n, V)$ and let $\rho$ be a finite-dimensional, rational representation of $\GL_n(\C)$ on $V$ of homogeneous degree $k$ such that $\rho$ respects the rational structure on $V$. For $u=\pm 1$, define the function $R^u_\rho(f)$ on $\H_n$ by $(R^u_\rho(f))(z) = \rho(y^{\frac{u}2})f(z)$. Then $R^u_\rho(f) \in  N_S(\H_n, V)$, and for each $\sigma \in \Aut(\C)$ we have \[\pi^{\frac{uk}{2}}R^u_\rho( {}^\sigma\!f) =  {}^\sigma\!(\pi^{\frac{uk}{2}}R^u_\rho(f)).\]
\end{lemma}
\begin{proof} By assumption, we can write
\[
 f(z) = \sum_h \!\left(\sum_a q_{h,a}((\pi  y)^{1/2}) v_a \right) \exp(2 \pi i\, {\rm tr}(hz)),
\]
where $q_{h,a} \in \C[\mathcal{T}^\pm]$ and the $v_a$ form a rational basis of $V$. Furthermore, from the assumptions on $\rho$ we have $\rho(y^{u/2}) v_a = \sum_b r_{u,a,b}(y^{1/2}) v_b$ where the $r_{u,a,b} \in \Q[\mathcal{T}^\pm]$ have homogeneous degree $uk$.  In particular $r_{u,a,b}((\pi  y)^{1/2}) = \pi^{uk/2} r_{u,a,b}(y^{1/2})$.
This gives us
\[
 \pi^{uk/2} R^u_\rho(f)(z) = \sum_h \!\Bigg(\sum_{a,b} q_{h,a}((\pi  y)^{1/2})  r_{u, a,b}((\pi  y)^{1/2}) v_b \Bigg) \exp(2 \pi i\, {\rm tr}(hz)),
\]
from which the assertion of the lemma is clear.
\end{proof}

We define $\U(\mathfrak k_\Q)$ to be the $\Q$-subalgebra of $\U(\mathfrak k_\C)$ generated by the various $B_{a_1, b_1}$ (see \eqref{BijEijdefeq}). The elements of $\U(\mathfrak k_\Q)$ are sums of the form $c_0 + \sum_{i=1}^n c_i B_{a^i_1, b^i_1} \cdots B_{a^{i}_{u_i}, b^{i}_{u_i}}$ where $n \ge 0$, $c_i \in \Q$ for $0 \le i \le n$, $u_i \ge 1$ for $1 \le i \le n$, and $1\le a^i_k, b^i_k \le n$ for $1\le k \le u_i$.

Recall that for each representation $(\rho, V)$ of $\GL_n(\C)$, the derived representation of $\U(\mathfrak k_\C)$ on $V$  is also denoted by $\rho$. For a general $X \in \U(\mathfrak k_\C)$, $\rho(X) \in \mathrm{End}(V)$ may not be invertible (i.e., $\rho(X)$ may not be in $\GL(V)$). However the following lemma is immediate.

\begin{lemma}\label{l:trivialX}Let $X \in \U(\mathfrak k_\Q)$ and let $\rho$ be a finite-dimensional representation of $\GL_n(\C)$ on $V$. Then for $i=1,2$, there exists $X_i \in \U(\mathfrak k_\Q)$ such that $X= X_1 +X_2$ and $\rho(\iota(X_i))$ is invertible.
\end{lemma}
\begin{proof} We choose a sufficiently large $c \in \Q$ such that $\rho(\iota(X)) + c I$ and $\rho(\iota(X)) - c I$ are both invertible, where $I$ denotes the identity map on $V$. Now the result follows from taking $X_1 = \frac12(X + c)$, $X_2 = \frac12(X - c)$.
\end{proof}

\begin{lemma}\label{l:Babrationalautc}Let $\rho$ be an irreducible rational representation of $\GL_n(\C)$ on $V$ such that $\rho$ respects the rational structure on $V$. Let $X \in \U(\mathfrak k_\Q)$  and  for each $f$ in $N_S(\H_n, V)$, define $\widetilde{X} f \in C^\infty(\H_n, V)$ by $(\widetilde{X} f)(z) = \rho(\iota(X)) f(z)$. Then $\widetilde{X} f \in N_S(\H_n, V)$ and for each $\sigma \in \Aut(\C)$, we have \[{}^\sigma\!(\widetilde{X} f) = \widetilde{X}({}^\sigma\! f).\]

Furthermore, suppose that $\rho(\iota(X))$ is invertible and let the representation $\rho'$  of $\GL_n(\C)$ on $V$  be given by $\rho'(h) = \rho(\iota(X))\rho(h)\rho(\iota(X))^{-1}$. Then  the following hold:
\begin{enumerate}
 \item  If $f \in N_\rho(\Gamma)$, then $\widetilde{X} f \in N_{\rho'}(\Gamma)$,
 \item $(\widetilde{X} f)^{\rho'} = Xf^{\rho}$,
 \item    $\rho'$ respects the rational structure on $V$.
\end{enumerate}
\end{lemma}
\begin{proof}
First of all, we note that
\begin{equation}\label{e:rationallieaction}  \rho(\iota(X)) \in  {\rm End}_\Q(V_\Q)\end{equation}
This follows from the definition of the rational structure  and the fact that the $\rho(B_{a_i, b_i})$ preserve $V_\Q$.

We  now show that $\widetilde{X} f \in N_S(\H_n, V)$ and that ${}^\sigma\!(\widetilde{X} f) = \widetilde{X}({}^\sigma\! f)$ for each $\sigma \in \Aut(\C)$. Since  $f\in N_S(\H_n, V)$, we can write
\[
 f(z) = \sum_h \!\left(\sum_a q_{h,a}((\pi  y)^{1/2}) v_a \right) \exp(2 \pi i\, {\rm tr}(hz)),
\]
where $q_{h,a} \in \C[\mathcal{T}^\pm]$ and the $v_a$ form a rational basis of $V$. Using \eqref{e:rationallieaction} we can write   $\rho(\iota(X)) v_a = \sum_b r_{a,b}v_b$ where the $r_{a,b} \in \Q$. This gives us
\[
 (\widetilde{X} f )(z) = \sum_h \!\Bigg(\sum_{a,b} q_{h,b}((\pi  y)^{1/2})  r_{a,b} v_b \Bigg) \exp(2 \pi i\, {\rm tr}(hz)),
\]
from which the required facts follow immediately.

Next, for each $z=x+iy \in \H_n$, put $g_z = \mat{y^{1/2}}{xy^{-1/2}}{}{y^{-1/2}}$. Using \eqref{frho-denf} and \eqref{e:liealgebraactionvec}, we see that
\[
 (\widetilde{X} f)^{\rho'}\left(g_z\right) = \rho'(y^{1/2}) (\widetilde{X}f)(z)= \rho(\iota(X)) \rho(y^{1/2}) f(z) = \rho(\iota(X)) f^{\rho} (g_z)=  Xf^{\rho}  \left(g_z\right).
\]
On the other hand,  using \eqref{e:ktranseq}, \eqref{e:liealgebraactionvec}   we observe that for each $g \in \Sp_{2n}(\R)$, $k \in K_\infty$, $(\widetilde{X} f)^{\rho'}(gk) = \rho'(\iota(k))(\widetilde{X} f)^{\rho'}(g)$, $Xf^{\rho}(gk) = \rho'(\iota(k))Xf^{\rho}(g)$. Since each element of $\Sp_{2n}(\R)$ can be written as a product of an element of the form $g_z$ and an element of $K_\infty$, it follows that $(\widetilde{X} f)^{\rho'} = Xf^{\rho}$.

Next, suppose that $f \in N_\rho(\Gamma)$. We will show then that $\widetilde{X} f \in N_{\rho'}(\Gamma)$. It is clear that $\widetilde{X} f$ is nearly holomorphic in this case, so we only need to show that $(\widetilde{X} f) |_{\rho'} \gamma = \widetilde{X} f$ for all $\gamma \in \Gamma$. This follows  from the calculation
\[
 (\widetilde{X} f)(\gamma z)=\rho(\iota(X)) f(\gamma z) = \rho(\iota(X))\rho (J(\gamma, z)) f(z) = \rho' (J(\gamma, z))  \widetilde{X} f(z).
\]
The assertion that $\rho'$ respects the rational structure on $V$ follows from \eqref{e:rationallieaction} and the fact that $\rho$ respects the rational structure on $V$.
\end{proof}
\begin{proposition}\label{p:master}
 Let $\rho_1$ and $\rho_2$ be finite-dimensional rational representations of $\GL_n(\C)$ on $V$ of homogeneneous degrees $k(\rho_1)$ and $k(\rho_2)$ respectively and admitting a common rational structure on $V$; assume also that $\rho_1$ is irreducible. Let $R \in \U(\g_\C)$ be of the form
 \[
  R =  \sum_i  c_i  E_{+,m_1,n_1}\cdots E_{+,m_{s_i}, n_{s_i}} E_{-,p_1, q_1}\cdots E_{-,p_{t_i}, q_{t_i}} B_{a_1, b_1} \cdots B_{a_{u_i}, b_{u_i}},
 \]
 where the $c_i$ are rational numbers and $s_i$, $t_i$, $u_i$ are non-negative integers. Let $\Gamma$ be a congruence subgroup of $\Sp_{2n}(\Q)$. For each $f$ in $N_{\rho_1}(\Gamma)$, define the function $\widetilde{R} f \in C^\infty(\H_n, V)$ by
 \begin{equation}\label{p:mastereq1}
  (\widetilde{R} f)(z) = \rho_2(y^{-1/2})\left((Rf^{\rho_1})\mat{y^{1/2}}{xy^{-1/2}}{}{y^{-1/2}}\right).
 \end{equation}
 Then $\widetilde{R} f \in N_S(\H_n, V)$ and
 \begin{equation}\label{p:mastereq2}
  \rule{0ex}{2.5ex}^\sigma\!\Big(\pi^{\frac{k(\rho_1)-k(\rho_2)}{2}}\widetilde{R} f\Big) = \pi^{\frac{k(\rho_1)-k(\rho_2)}{2}}\widetilde{R} ({}^\sigma\!f)
 \end{equation}
 for all $\sigma \in \Aut(\C)$.
\end{proposition}
\begin{proof}
By linearity and using Lemma \ref{l:trivialX}, it suffices to consider the case
$$
 R= E_{+,m_1,n_1}\cdots E_{+,m_{s}, n_{s}} E_{-,p_1, q_1}\cdots E_{-,p_{t}, q_{t}} X
$$
where $X \in \U(\mathfrak k_\Q)$ and $\rho_1(\iota(X))$ is invertible. Put $E=  E_{+,m_1,n_1}\cdots E_{+,m_{s}, n_{s}} E_{-,p_1, q_1}\cdots E_{-,p_{t}, q_{t}}$, $g_f = \rho_1(\iota(X)) f$, and $\rho'(h) = \rho_1(\iota(X))\rho_1(h)\rho_1(\iota(X))^{-1}$. For each $z=x+iy \in \H_n$, put $g_z = \mat{y^{1/2}}{xy^{-1/2}}{}{y^{-1/2}}$. By Lemma~\ref{l:Babrationalautc},
\begin{equation}\label{e:master1}
 (\widetilde{R} f)(z) = \rho_2(y^{-1/2}) \left((Eg_f^{\rho'})(g_z)\right)\end{equation}
and
\begin{equation}\label{e:master2}
 g_f \in N_{\rho'}(\Gamma), \quad {}^\sigma\!(g_f) = g_{({}^\sigma\!f)}.
\end{equation}
For $1\le a \le s$, $1 \le b \le t$, let $u_a$ and $v_b$ be elements such that $\iota_+(u_a) = E_{+,m_a,n_a}$,  $\iota_-(v_b) = E_{-,p_b,q_b}$.
By (\ref{iota-defn}) and (\ref{BijEijdefeq}), we see that $iu_a, iv_b \in M_n^{\rm sym}(\Q)$. Using part 2 of Proposition \ref{alg-op-reln-prop} and Proposition \ref{shimura}, we see that there exists an operator $\mathcal{D}_E: C^\infty(\H_n, V) \rightarrow C^\infty(\H_n, \mathrm{Ml}_{s+t}(T, V))$ such that \begin{equation}\label{e:master3}
{}^\sigma\!((\pi i)^{t-s}\mathcal{D}_E g_f) =(\pi i)^{t-s}\mathcal{D}_E ({}^\sigma\!g_f) =  (\pi i)^{t-s}\mathcal{D}_E (g_{{}^\sigma\!f})\end{equation}
and such that the following identities hold
\begin{align*}
 (Eg_f^{\rho'})(g_z)&= (\iota_+(u_1)\ldots \iota_+(u_s)\iota_-(v_1) \ldots \iota_{-}(v_t)g_f^{\rho'})(g_z)
 \label{eq:1st}
 \\ &=\left((\mathcal{D}_E g_f)^{\rho' \otimes \tau^s \otimes \sigma^t}(g_z)\right)(u_1,\ldots, u_s, v_1, \ldots, v_t)\\
 &=\left((\rho'\otimes\tau^s\otimes\sigma^t)(y^{1/2})
 (\mathcal{D}_E g_f)(z)\right)(u_1,\ldots, u_s, v_1, \ldots, v_t)\\
 &=\rho'(y^\frac12)\left((\mathcal{D}_E g_f)(z)(y^\frac12u_1y^\frac12, \ldots y^\frac12u_sy^\frac12, y^{-\frac12}v_1y^{-\frac12}, \ldots, y^{-\frac12}v_ty^{-\frac12} )\right).
\end{align*}
Combining this with \eqref{e:master1} and using multilinearity we obtain
\begin{equation}\label{e:master4} \textstyle
 (\widetilde{R} f)(z) = \rho_2(y^{-\frac12})\rho'(y^\frac12) \left(((\pi i)^{t-s}\mathcal{D}_E g_f)(z)\left((\pi  y)^\frac12(i u_1)(\pi  y)^\frac12, \ldots, (\pi  y)^{-\frac12}(-i v_t)(\pi  y)^{-\frac12} \right)\right).
\end{equation}
Now, combining \eqref{e:master4} with Lemma \ref{l:trivaction}, Lemma \ref{l:rationalityrho}, and \eqref{e:master3}, we obtain the desired result.
\end{proof}
\subsection{Arithmeticity of certain differential operators on nearly holomorphic cusp forms} \label{s:proj}\label{s:nearholo}
For each $\k \in \Lambda^+$, let $\chi_\k$ be the character by which $\mathcal{Z}$ acts on $L(\k)$. (Note that this notation differs from what we used in \cite{PSS14};  $\chi_\k$ was denoted $\chi_{\k + \varrho}$ there.)
\begin{lemma}\label{near-holo-arch-lemma}
 Let $\k = (k_1, \ldots, k_n)$ and $\k'=(k'_1, \ldots, k'_n)$ be elements of $\Lambda^+$ such that either $k_n,k'_n\geq n$ or $k_n+k'_n>2n$. Then $\chi_{\k'} = \chi_\k$ if and only if $\k'=\k$.
\end{lemma}
\begin{proof}
Suppose that $\chi_\k = \chi_{\k'}$. Then $\k' = w \cdot \k$ for a Weyl group element $w$, where $\cdot$ denotes the dot action of the Weyl group (see Sections 1.8 to 1.10 of \cite{Humphreys2008}). The element $w$ decomposes as $w = \tau_I \sigma$ with $\sigma \in \mathfrak S_n$, the symmetric group for $\{1, \ldots, n\}$, and $I$ a subset of $\{1, \ldots, n\}$. The action $\tau_I \sigma \cdot \k$ is given by
$$
 \tau_I \sigma \cdot \k = \big(\epsilon_1 (k_{\sigma(1)}-\sigma(1)), \ldots, \epsilon_n (k_{\sigma(n)}-\sigma(n))\big) + (1, \ldots, n),
$$
where $\epsilon_i = 1$  if  $i \notin I$ and $\epsilon_i =-1$  if  $i \in I$. (Note that the dot action is given by $w\cdot\k=w(\k+\varrho)-\varrho$, where $\varrho$ is half the sum of the positive roots. In order to be in the setting of \cite{Humphreys2008}, the positive roots have to be $e_i-e_j$ for $1\leq i<j\leq n$ and $-(e_i+e_j)$ for $1\leq i\leq j\leq n$, resulting in $\varrho=-(1,\ldots,n)$). Hence
\begin{equation}\label{near-holo-arch-lemmaeq1}
 \epsilon_j (k_{\sigma(j)}-\sigma(j))=k'_j-j\qquad\text{for }j\in\{1,\ldots,n\}.
\end{equation}
Let $S=\{j\in\{1,\ldots,n\}\ |\ \sigma(j)\neq j\}$. Then $\sigma$ induces a permutation of $S$ without fixed points. The condition \eqref{near-holo-arch-lemmaeq1}, together with $k_j+k'_j\geq 2n$, forces $\epsilon_j=1$ for all $j\in S$. Hence $x_{\sigma(j)}=x'_j$ for $j\in S$, where $x_i:=k_i-i$ and $x'_i=k'_i-i$. Since the $x_i$ and the $x'_i$ are strictly ordered from largest to smallest, this is only possible if $S=\emptyset$.

We proved that $\sigma$ is the identity, so that $\epsilon_j(k_j-j)=k'_j-j$ for all $j\in\{1,\ldots,n\}$. If $\epsilon_j=1$, then $k_j=k'_j$. If $\epsilon_j=-1$, then our hypothesis implies $j=n$ and $k_j=k'_j=n$. This proves $\k=\k'$.
\end{proof}

For $1 \leq i \leq n$, let $D_{2i}$ be the generators of $\mathcal Z$ from Theorem \ref{t:Mau12}. Note that $\mathcal{Z}$, and hence each $D_{2i}$, acts on the space $\AA(\Gamma; \rho)^\circ_{\p_-\text{-fin}}$. By the isomorphism $N_{\rho}(\Gamma) \isomto \AA(\Gamma; \rho)_{\pp}$ of Proposition~\ref{p:iso}, it follows that $D_{2i}$ gives rise to an operator $\Omega_{2i}$ on the space  $N_\rho(\Gamma)$. Precisely, for each $f \in N_\rho(\Gamma)$,
\[
 (\Omega_{2i} f)^\rho = D_{2i}(f^\rho), \quad L_\rho((\Omega_{2i} f)^\rho) = D_{2i}(L_\rho(f^\rho)).
\]
More concretely, for $z=x+iy \in \H_n$ and $f \in N_\rho(\Gamma)$, the above definition is equivalent to \begin{equation}\label{e:opertorsdef}
 (\Omega_{2i} f)(z) = \rho(y^{-1/2})\left((D_{2i}f^\rho)\mat{y^{1/2}}{xy^{-1/2}}{}{y^{-1/2}}\right).
\end{equation}

For each character $\chi$ of $\mathcal{Z}$, we define $\AA(\Gamma; \rho, \chi)^\circ_{\p_-\text{-fin}}\subseteq \AA(\Gamma; \rho)_{\p_-\text{-fin}}^\circ$ to be the subspace consisting of all the elements on which $\mathcal{Z}$ acts via the character $\chi$. We let $N_\rho(\Gamma; \chi)^\circ$ be the corresponding subspace of $N_\rho(\Gamma)^\circ$. In the notation introduced in Sect.~\ref{s:ident}, it is clear that
\begin{equation}\label{e:decom3}
 \AA(\Gamma; \rho,  L(\k))^\circ\subseteq \AA(\Gamma; \rho, \chi_\k)_{\pp}^\circ,\qquad N_\rho(\Gamma; L(\k))^\circ \subseteq N_\rho(\Gamma; \chi_\k)^\circ.
\end{equation}
More precisely,
\begin{equation}\label{e:decom4}
 \bigoplus_{\substack{\k'\in\Lambda^{++}\\\chi_{\k'}=\chi_\k}}\AA(\Gamma; \rho, L(\k'))^\circ=\AA(\Gamma; \rho, \chi_\k)_{\pp}^\circ,\qquad\bigoplus_{\substack{\k'\in\Lambda^{++}\\\chi_{\k'}=\chi_\k}} N_\rho(\Gamma; L(\k'))^\circ=N_\rho(\Gamma; \chi_\k)^\circ.
\end{equation}
It follows from Lemma \ref{near-holo-arch-lemma} that
\begin{equation}\label{e:decom5}
 N_\rho(\Gamma; L(\k))^\circ=N_\rho(\Gamma; \chi_\k)^\circ \qquad\text{if }\k=(k_1,\ldots,k_n)\in\Lambda^{++}\text{ with }k_n\geq2n,
\end{equation}
and similarly for $\AA(\Gamma; \rho, \chi_\k)_{\pp}^\circ$.

\begin{lemma}\label{l:integral}
Let $\k \in \Lambda^+$. Then for all $1 \leq i \leq n$, $\chi_\k(D_{2i}) \in \Z$.
\end{lemma}
\begin{proof}
Since $D_{2i} v = \chi_\k(D_{2i}) v$ for any $v$ in the space of $L(\k)$, it is enough to prove that $D_{2i}v_\k$ is an \emph{integral} multiple of $v_\k$, where we fix $v_\k$ to be a highest weight vector of weight $\k$ in $L(\k)$. For this we appeal to Corollary \ref{Center-gener-cor}. We know that $v_\k$ is annihilated by all the $E_{-,*,*}$ elements, as well as all the $B_{a,b}$ elements for $a>b$, and the number of $E_{-,*,*}$ elements is the same as the number of $E_{+,*,*}$ elements in the expression of $D_{2i}$. Hence, we may write $D_{2i}v_\k =  \sum_{i=1}^t c_i H^{(i)}_{1}\cdots H^{(i)}_{r_i} X^{(i)}_1\cdots X^{(i)}_{p_i} v_\k$ where the $c_i$ are integers, $r_i, p_i \ge 0$ and each $H^{(i)}_k$ is equal to $B_{a,a}$ for some $a$, and each $X^{(i)}_k$ is equal to $B_{a,b}$ for some $a < b$. We break up the expression above into two parts corresponding to whether $p_i=0$ or $p_i>0$, and obtain $D_{2i}v_\k = w_1 + w_2$ where $w_1 = \sum_{i=1}^u c_i H^{(i)}_{1}\cdots H^{(i)}_{r_i} v_\k$ and $w_2 = \sum_{i=u+1}^t c_i H^{(i)}_{1}\cdots H^{(i)}_{m_i} X^{(i)}_1\cdots X^{(i)}_{n_i} v_\k$ with each $n_i>0$. Looking at weights, we see $w_2=0$. Since $B_{a,a} v_\k = k_a v_\k$, it follows that $w_1$ is an integral multiple of $v_\k$, completing the proof.
\end{proof}

\begin{lemma}\label{l:omegaarith}  Let $\sigma \in \Aut(\C)$. Then for any $f \in N_\rho(\Gamma)$
$${}^\sigma\!(\Omega_{2i} f) = \Omega_{2i}({}^\sigma\!f).$$
\end{lemma}
\begin{proof}
This follows from Corollary \ref{Center-gener-cor2}, Proposition \ref{p:master} and \eqref{e:opertorsdef}.
\end{proof}

For each $\k \in \Lambda^{++}$, let $\p_{\chi_\k}$ denote the orthogonal projection map from $N_{\rho}(\Gamma)^\circ$ to its subspace $N_{\rho}(\Gamma; \chi_\k)^\circ$. (We omit $\rho$ from the notation of $\p_{\chi_\k}$ for brevity.) If $\k = (k_1, \ldots, k_n) \in \Lambda^+$ such that $k_n \ge 2n$, then \eqref{e:decom5} implies that $\p_{\chi_\k}$ is precisely the orthogonal projection map from $N_{\rho}(\Gamma)^\circ$ to its subspace $N_{\rho}(\Gamma; L(\k))^\circ$.

\begin{proposition}\label{p:projarith}
 For any $\sigma \in \Aut(\C)$, $f \in N_{\rho}(\Gamma)^\circ$, and $\k\in \Lambda^{++}$, we have
 $$
  {}^\sigma(\p_{\chi_\k}(f)) = \p_{\chi_\k}(\,{}^\sigma\!f).
 $$
\end{proposition}
\begin{proof}
We can give an explicit formula for $\p_{\chi_\k}$ as follows,
\begin{equation}\label{exp-defn-proj}
\p_{\chi_\k} = C_\k^{-1} \prod\limits_{\substack{\k' \in \Lambda^{++}\\ \k'\neq\k}} \:\sum\limits_{i=1}^n {\rm sgn}\big(\chi_\k(D_{2i}) - \chi_{\k'}(D_{2i})\big) \big(\Omega_{2i} - \chi_{\k'}(D_{2i})\big),
\end{equation}
where
\begin{equation}\label{exp-defn-proj2}
 C_\k = \prod\limits_{\substack{\k' \in \Lambda^{++}\\ \k'\neq\k}} \:\sum\limits_{i=1}^n |\chi_\k(D_{2i}) - \chi_{\k'}(D_{2i})|.
\end{equation}
In both \eqref{exp-defn-proj} and \eqref{exp-defn-proj2}, the product extends only over those finitely many $\k'$ for which the space $N_\rho(\Gamma;\chi_\k)^\circ$ is non-zero. Now the proposition follows from Lemmas \ref{l:integral} and \ref{l:omegaarith}.
\end{proof}

In the following we denote by $\mathfrak{k}_\Q$ the Lie algebra over $\Q$ spanned by the elements $B_{ij}$ defined in \eqref{BijEijdefeq}. Then $\mathfrak{k}_\C=\mathfrak{k}_\Q\otimes_\Q\C$.

\begin{lemma}\label{rhokrationallemma}
 For $\mathbf{k}\in\Lambda^+$, let $F_{\mathbf{k}}$ be a model for $\rho_{\mathbf{k}}$, and let $v_{\mathbf{k}}$ be a highest weight vector in $F_{\mathbf{k}}$. Let $F_{\mathbf{k},\Q}:=\U(\mathfrak{k}_\Q)v_{\mathbf{k}}$. Then $F_{\mathbf{k},\Q}$ is an irreducible $\mathfrak{k}_\Q$-module, and $F_{\mathbf{k},\Q}\otimes_\Q\C\cong F_{\mathbf{k}}$ as $\mathfrak{k}_\C$-modules.
\end{lemma}
\begin{proof}
By Weyl's theorem \cite[Theorem 7.8.11]{Weibel94}, $F_{\mathbf{k},\Q}$ is a direct sum $U_1\oplus\ldots\oplus U_m$ of irreducible $\mathfrak{k}_\Q$-modules. Evidently, $v_{\mathbf{k}}$ must have a non-zero component $v_i$ in each of the $U_i$. Then each $v_i$ has weight $\mathbf{k}$. Since, up to multiples, $v_{\mathbf{k}}$ is the only vector of weight $\mathbf{k}$ in $F_{\mathbf{k}}$, there exist $c_i\in\C$ with $v_i=c_iv_{\mathbf{k}}$. We can get from $v_{\mathbf{k}}$ to any other vector in $F_{\mathbf{k},\Q}$ with an appropriate element of $\U(\mathfrak{k}_\Q)$. In particular, there exists $X_i\in\U(\mathfrak{k}_\Q)$ with $X_iv_{\mathbf{k}}=v_i$. Looking at weights, we see that $X_i$ must be a constant, so that in fact $X_i=c_i$. It follows that all the $c_i$ are rational. On the other hand, the $v_i$ are $\Q$-linearly independent. This is only possible if $m=1$.

We proved the irreducibility assertion. It follows that $F_{\mathbf{k},\Q}\otimes_\Q\C$ is an irreducible $\mathfrak{k}_\C$-module. The obvious map from $F_{\mathbf{k},\Q}\otimes_\Q\C$ to $F_{\mathbf{k}}$ is then an isomorphism.
\end{proof}

Let $\p_{\pm,\Q}$ be the $\Q$-span of the elements $E_{\pm,i,j}$ defined in \eqref{BijEijdefeq}; then $\p_{\pm,\Q}$ is a Lie algebra over $\Q$. Set $\mathfrak{g}_\Q=\p_{+,\Q}\oplus\mathfrak{k}_\Q\oplus\p_{-,\Q}$. Then $\mathfrak{g}_\Q$ is a rational form of $\mathfrak{g}_\C$.

\begin{lemma}\label{Nkrationallemma}
 For $\mathbf{k}\in\Lambda^+$, recall the $\g_\C$-module $N(\k)=\U(\g_\C) \otimes_{\U(\mathfrak{k}_\C + \p_{-})} F_\k$. Consider the $\g_\Q$-module $N(\k)_\Q:=\U(\g_\Q)\otimes_{\U(\mathfrak{k}_\Q+ \p_{-,\Q})} F_{\k,\Q}$. Then $N(\k)\cong N(\k)_\Q\otimes_\Q\C$ as $\g_\C$-modules. There exists a direct sum decomposition of $N(\k)_\Q$ into irreducibles $U_j$ such that the $U_j\otimes_\Q\C$ are the $K_\infty$-types of $N(\k)$.
\end{lemma}
\begin{proof}
Inside $\U(\p_+)$, let $\U(\p_+)^{(m)}$ be the $\C$-linear span of the elements $E_{+,i_1,j_1}\ldots E_{+,i_m,j_m}$, and let $\U(\p_+)^{(m)}_\Q$ be the $\Q$-linear span of the same elements. Using the PBW theorem, we have
\begin{equation}\label{Nkrationallemmaeq1}
 N(\k)=\U(\g_\C) \otimes_{\U(\mathfrak{k}_\C + \p_{-})} F_\k\cong\U(\p_+)\otimes_\C F_\k\cong\bigoplus_{m=0}^\infty\Big(\U(\p_+)^{(m)}\otimes_\C F_\k\Big)
\end{equation}
as complex vector spaces. It follows from the PBW theorem and Lemma \ref{rhokrationallemma} that
\begin{equation}\label{Nkrationallemmaeq2}
 \U(\p_+)^{(m)}\otimes_\C F_\k\cong(\U(\p_{+,\Q})^{(m)}\otimes_\Q F_{\k,\Q})\otimes_\Q\C
\end{equation}
(both sides have the same $\C$-dimension). Hence
\begin{align}\label{Nkrationallemmaeq3}
 N(\k)&\cong\bigg(\bigoplus_{m=0}^\infty\Big(\U(\p_{+,\Q})^{(m)}\otimes_\Q F_{\k,\Q}\Big)\bigg)\otimes_\Q\C\nonumber\\
 &\cong\Big(\U(\p_{+,\Q})\otimes_\Q F_{\k,\Q}\Big)\otimes_\Q\C\nonumber\\
 &\cong\Big(\U(\g_\Q)\otimes_{\U(\mathfrak{k}_\Q+ \p_{-,\Q})} F_{\k,\Q}\Big)\otimes_\Q\C.
\end{align}
This proves $N(\k)\cong N(\k)_\Q\otimes_\Q\C$ as $\C$-vector spaces. It is easy to see that the isomorphism is compatible with the action of $\g_\C\cong\g_\Q\otimes_\Q\C$.

It follows from $[\mathfrak{k}_\Q,\p_{+,\Q}]=\p_{+,\Q}$ that $\U(\p_{+,\Q})^{(m)}\otimes_\Q F_{\k,\Q}$ is a $\mathfrak{k}_\Q$-module. By Weyl's theorem, we may decompose it into irreducibles $U_j$. Then the $U_j\otimes_\Q\C$ are irreducible under the action of $\mathfrak{k}_\C\cong\mathfrak{k}_\Q\otimes\C$, i.e., they are the $K_\infty$-types of $\U(\p_+)^{(m)}\otimes_\C F_\k$.
\end{proof}

\begin{lemma}\label{Lkrationallemma}
 For $\mathbf{k}\in\Lambda^+$, let $v_\k$ be a highest weight vector in the $K_\infty$-type $\rho_\k$ of $L(\k)$. Then for any $K_\infty$-type $\rho$ of $L(\k)$ there exists a non-negative integer $m$ such that $\rho$ admits a $\C$-basis consisting of vectors of the form $Y_iv_\k$, where $Y_i\in\U(\p_{+,\Q})^{(m)}\U(\mathfrak{k}_\Q)$.
\end{lemma}
\begin{proof}
Let $U_j\otimes_\Q\C$ be one of the $K_\infty$-types from Lemma \ref{Nkrationallemma} such that it maps onto $\rho$ under the projection $N(\k)\to L(\k)$. By construction there exists an $m$ such that $U_j$ has a basis consisting of vectors of the form $Y_iv_0$ with $Y_i\in\U(\p_{+,\Q})^{(m)}\U(\mathfrak{k}_\Q)$; here $v_0$ is a highest weight vector in $F_{\k,\Q}$. The vectors $Y_iv_0\otimes1=Y_i(v_0\otimes1)$ are a $\C$-basis of $U_j\otimes_\Q\C$. Our assertion follows since, after proper normalization, $v_0\otimes1$ projects to $v_\k$.
\end{proof}

For the rest of this subsection let ${\mathbf{k}} = (k_1, k_2, \ldots, k_n)$ where  $k_1\ge\ldots\ge k_n>n$ are integers with the \emph{same parity}. In this case, $L(\k)$ is a \emph{holomorphic discrete series representation} that contains the $K_\infty$-type\footnote{The $K_\infty$-type $\rho_{k_1, k_1, \ldots, k_1}$  corresponds to the character $\det(J(k_\infty, iI_n))^{-k_1}$ of $K_\infty$, or equivalently, the character $\det^{k_1}$ of $U(n)$.} $\rho_{k_1, k_1, \ldots, k_1}$ with multiplicity one (see Lemma 5.3 of \cite{PSS17}). We use the notation $N_{k_1}(\Gamma; L(\k))^\circ:=N_{\det^{k_1}}(\Gamma; L(\k))^\circ$.  We will construct a differential operator that maps $S_{\mathbf{k}}(\Gamma)$ isomorphically onto $N_{k_1}(\Gamma; L(\k))^\circ$. The idea is to find a $\Q$-rational element of $\mathcal{U}(\mathfrak{g}_{\C})$ that maps a highest-weight vector of weight $\k$ inside $L(\k)$ onto a vector in the one-dimensional, multiplicity one $K_\infty$-type $\rho_{k_1, k_1, \ldots, k_1}$ in $L(\k)$. This can be done thanks to Lemma \ref{Lkrationallemma}.

\begin{proposition}\label{p:uoper}
 Let ${\mathbf{k}} = (k_1, k_2, \ldots, k_n)$ where  $k_1\ge\ldots\ge k_n>n$ are integers with the same parity. Then there exists an injective linear map $\mathcal{D}_{\mathbf{k}}$ from $S_{\mathbf{k}}(\Gamma)$ to $N_{k_1}(\Gamma)^\circ$ with the following properties.
 \begin{enumerate}
  \item The image $\mathcal{D}_{\mathbf{k}}(S_{\mathbf{k}}(\Gamma))$ is equal to $N_{k_1}(\Gamma; L(\k))^\circ$.
  \item For any $\sigma \in \mathrm{Aut}(\C)$ and $f \in S_{\mathbf{k}}(\Gamma)$ we have $\mathcal{D}_{\mathbf{k}}(\,{}^\sigma\!f)= {}^\sigma(\mathcal{D}_{\mathbf{k}} f)$.
 \end{enumerate}
\end{proposition}
\begin{proof}
For brevity, we write $\rho = \rho_\k$. Let $f \in S_{\mathbf{k}}(\Gamma)$, and put $f' = L_\rho\circ f^\rho$. By Proposition \ref{AA0nfindecomp2prop}, $f'$ is a highest weight vector of weight $\k$  inside  $V = \U(\g_\C)(f')\simeq L(\k)$.

By Lemma \ref{Lkrationallemma}, there exists an element   $Y \in \U(\p_{+,\Q})^{(m)}\U(\mathfrak{k}_\Q)$ such that when $Y$ is viewed as an operator on $V$, then $Y v_\k$ is a non-zero vector in the one-dimensional $K_\infty$-type $\rho_2:=\rho_{k_1, k_1, \ldots, k_1}$.  Denote $\widetilde{Y} = \pi^{\frac{k_1 + \ldots + k_n - nk_1}2}  Y$ and define the map $\mathcal D_{\k}$ from $S_{\mathbf{k}}(\Gamma)$ to $C^\infty(\H_n)$ by
\begin{align}\label{e:finalarithmetic}
 (D_{\k}f)(z) &:= \det(y)^{-k_1/2} \widetilde{Y}(L_\rho(f^\rho))(g_z)\nonumber\\
  & = L_\rho\Big(\det(y)^{-k_1/2} (\widetilde{Y}(f^\rho))(g_z)\Big).
\end{align}
From the construction and the aforementioned property of $Y$, it is clear that  $(D_{\k}f)^{\rho_2} =   L_\rho (\widetilde{Y} f^\rho) = \widetilde{Y}(L_\rho f^\rho)$. This, together with the fact that $\rho_2$ is one-dimensional and occurs with multiplicity one in $L(\k)$ (and using Propositions \ref{p:iso} and \ref{AA0nfindecomp2prop}), it follows that $\mathcal D_{\k}$ maps $S_{\mathbf{k}}(\Gamma)$ surjectively onto $N_{k_1}(\Gamma; L(\k))^\circ$. This proves part 1. Using Proposition \ref{p:master} and the expression \eqref{e:finalarithmetic} (and the fact that the projection map $L_\rho$ is rational) we obtain part 2.
\end{proof}

\begin{remark}In the case $n=2$, we may take $\mathcal{D}_{\mathbf{k}} =  \pi^{-\frac{k_1 - k_2}2}U^{\frac{k_1 - k_2}2}$ using the notation of \cite{PSS14} (see Propositions 3.15 and 5.6 of \cite{PSS14}).
\end{remark}
\section{Special values of \texorpdfstring{$L$}{}-functions}

Throughout this section, we fix an $n$-tuple $\k=(k_1, k_2, \ldots, k_n)$ of positive integers such that $k_1\ge\ldots\ge k_n\ge n+1$ and all $k_i$ have the same parity. For brevity we write $k=k_1$.
\subsection{Automorphic representations}\label{s:autrep}
We recall that $L(\k)$ is the holomorphic discrete series representation of $\Sp_{2n}(\R)$ with highest weight $(k_1, k_2, \ldots, k_n)$. Let $\GSp_{2n}(\R)^+$ be the index two subgroup of $\GSp_{2n}(\R)$ consisting of elements with positive multiplier. We may extend $L(\k)$ in a trivial way to $\GSp_{2n}(\R)^+\cong\Sp_{2n}(\R)\times\R_{>0}$. This extension induces irreducibly to $\GSp_{2n}(\R)$. We denote the resulting representation of $\GSp_{2n}(\R)$ by the same symbol $L(\k)$.

Let $p$ be a prime and let $\sigma$ be an irreducible spherical representation of $\GSp_{2n}(\Q_p)$. Let $b_0$, $b_1,\ldots b_n$ be the  Satake parameters associated to $\sigma$ (see, e.g., \cite[\S 3.2]{ASch}). For each character $\chi$ of $\Q_p^\times$, put $\alpha_\chi = \chi(\varpi)$ for any uniformizer $\varpi$ of $\Q_p$ if $\chi$ is unramified, and put $\alpha_\chi=0$  if $\chi$ is ramified.  Define the local $\GSp_{2n} \times \GL_1$ standard $L$-function \[L(s,\sigma \boxtimes \chi):= (1 - \alpha_\chi p^{-s})^{-1} \prod_{i=1}^n \left((1-b_i \alpha_\chi p^{-s}) (1-b^{-1}_i \alpha_\chi p^{-s}) \right)^{-1}.\] Given another irreducible spherical representation $\sigma'$ of $\GSp_{2n}(\Q_p)$ with Satake parameters $b'_0$, $b'_1,\ldots b'_n$ we say that $\sigma \sim \sigma'$ if $(b_1, \ldots, b_n)$ and $(b'_1, \ldots, b'_n)$ represent the same tuple under the action of the Weyl group. Using Lemma 2.4 of \cite{Gelbart-Knapp} for $G = \GSp_{2n}(\Q_p)$, $H=Z(\Q_p)\Sp_{2n}(\Q_p)$, we see that the following conditions are equivalent:
     \begin{enumerate}
     \item $\sigma \sim \sigma'$,
     \item there is an unramified character $\chi$ of $\Q_p^\times$ such that $\sigma' \simeq \sigma \otimes (\chi \circ \mu_n)$,
      \item   there exists an irreducible admissible representation of $\Sp_{2n}(\Q_p)$ that occurs as a direct summand inside both $\sigma|_{\Sp_{2n}(\Q_p)}$ and $\sigma'|_{\Sp_{2n}(\Q_p)}$,
         \item $\sigma |_{\Sp_{2n}(\Q_p)} \simeq \sigma'|_{\Sp_{2n}(\Q_p)}$.
               \end{enumerate}

    Given any character $\chi$ of $\Q_p^\times$, and irreducible spherical representations $\sigma$, $\sigma'$ of $\GSp_{2n}(\Q_p)$ we  have \begin{equation}\label{e:simrelationLfunction}\sigma \sim \sigma' \quad \Rightarrow  \quad L(s,\sigma \boxtimes \chi) = L(s,\sigma' \boxtimes \chi). \end{equation}

We let $\mathcal{R}_{\k}(N)$ denote the set of irreducible cuspidal automorphic representations $\Pi\simeq\otimes \Pi_v$ of $\GSp_{2n}(\A)$ such that $\Pi_\infty \simeq L(\k)$ and such that $\Pi$ has a vector right invariant under the principal congruence subgroup $K_{2n}(N)$ of $\GSp_{2n}(\hat{\Z})$. Note that if $\Pi \in \mathcal{R}_{\k}(N)$ then $\Pi_p$ is spherical for all $p \nmid N$.
For $\Pi_1, \Pi_2 \in \mathcal{R}_{\k}(N)$, we define an equivalence relation $\Pi_1 \sim \Pi_2$ if $\Pi_{1, p} \sim \Pi_{2, p}$ for all $p \nmid N$. We define $$\widetilde{\mathcal{R}}_{\k}(N) =  \mathcal{R}_{\k}(N) / \sim.$$ For $\Pi \in \mathcal{R}_{\k}(N)$, let $[\Pi]$ denote its equivalence class in $\widetilde{\mathcal{R}}_{\k}(N)$.  For $\Pi_1, \Pi_2 \in \mathcal{R}_\k(N)$, and $\sigma \in \Aut(\C)$, it follows from the definition that if $\Pi_1 \sim \Pi_2$ then ${}^\sigma\!\Pi_1 \sim {}^\sigma\!\Pi_2$; therefore for any class $[\Pi] \in \widetilde{\mathcal{R}}_{\k}(N)$, the element ${}^\sigma\![\Pi] = [{}^\sigma\! \Pi] \in \widetilde{\mathcal{R}}_{\k}(N)$ is well-defined.  For $[\Pi] \in \widetilde{\mathcal{R}}_{\k}(N)$, and any set $S$ of places of $\Q$ such that $\infty \in S$ and $\Pi_p$ is unramified for $p \notin S$, define  $L^S(s,[\Pi] \boxtimes \chi) = \prod_{p\notin S} L(s,\Pi_p \boxtimes \chi_p)$ (this is well-defined by \eqref{e:simrelationLfunction}). If the set of finite primes in $S$ coincides with the set of primes dividing $N$, we will use the notation $L^N(s,[\Pi] \boxtimes \chi) := L^S(s,[\Pi] \boxtimes \chi)$. Given $[\Pi] \in \widetilde{\mathcal{R}}_{\k}(N)$, we let $\Q([\Pi])$  be the  fixed field of the set of $\sigma \in \Aut(\C)$ satisfying ${}^\sigma\![\Pi] = [\Pi]$.  Since $\Q([\Pi])$ is contained in the field of rationality $\Q(\Pi)$ of any automorphic representation in its equivalence class, it follows that $\Q([\Pi])$ is a totally real or CM field \cite[Theorems 3.2.1 and 4.4.1]{blaharam}.
\subsection{Classical and adelic cusp forms}\label{s:classicaladelic}
By the strong approximation theorem, we have \begin{equation}\label{e:strong} \GSp_{2n}(\A) = \bigsqcup_{\substack{1 \le a <N \\ (a, N)=1}}\GSp_{2n}(\Q) Z(\R)^+ \Sp_{2n}(\R) d_a K_{2n}(N)\end{equation} where $d_a \in \GSp_{2n}(\A_\f)$ is given by $(d_a)_p = \mat{I_n}{}{}{aI_n}$ if $p|N$, and $(d_a)_p = I_{2n}$ otherwise.

For each $h \in \GSp_{2n}(\A_\f)$, and each $\phi: \GSp_{2n}(\A) \rightarrow \C$, define $h\cdot\phi: \GSp_{2n}(\A) \rightarrow\C$ by $(h \cdot \phi)(g):= \phi(gh)$. For any compact open subgroup $U$ of $\GSp_{2n}(\widehat{\Z})$, let $$\Gamma_U = \Sp_{2n}(\Q)\cap U.$$ Let $\rho$ be a $K_\infty$-type that occurs in $L(\k)$. Then we define $\AA_{\GSp_{2n}(\A)}(U; \rho, L(\k))^\circ$ to be the space of functions $\phi$ from  $\GSp_{2n}(\A)$ to $\C$ such that
\begin{enumerate}

\item $\phi$ is a cuspidal automorphic form on $\GSp_{2n}(\A)$,

\item $\phi(g_\Q gz_\infty u) = \phi(g)$ for all $g \in \GSp_{2n}(\A)$, $u \in U$, $z_\infty \in Z(\R)^+$, $g_\Q \in \GSp_{2n}(\Q)$,

\item For each $h \in \GSp_{2n}(\A_\f)$,  the function $(h \cdot \phi)|_{\Sp_{2n}(\R)}$ lies in $\AA(\Gamma_{hUh^{-1}}; \rho, L(\k))^\circ$.

\end{enumerate}

It is clear that given any irreducible cuspidal automorphic representation $\Pi$ of $\GSp_{2n}(\A)$, we have \[\Pi \in \mathcal{R}_\k(N) \Longleftrightarrow V_\Pi \cap \AA_{\GSp_{2n}(\A)}(K_{2n}(N); \rho, L(\k))^\circ \neq \{0\}.\] For each $\phi \in \AA_{\GSp_{2n}(\A)}(U; \rho, L(\k))^\circ$, let $F_\phi:\H_n \mapsto V_\rho$ be the function corresponding to $\phi|_{\Sp_{2n}(\R)}$ under the isomorphism given by Proposition \ref{p:iso}. Then we have $F_\phi \in N_\rho(\Gamma_U; L(\k))^\circ$.
Now, consider $U = K_{2n}(N)$ which is a subgroup normalized by each $d_a$. Define for  $\phi \in \AA_{\GSp_{2n}(\A)}(K_{2n}(N); \rho, L(\k))^\circ$ and each $1 \le a <N$, $(a, N)=1$, the function $F^{(a)}_\phi:\H_n \mapsto V_\rho$ via $F^{(a)}_\phi:= F_{d_a\cdot\phi}$, where $d_a\cdot\phi \in \AA_{\GSp_{2n}(\A)}(K_{2n}(N); \rho, L(\k))^\circ$ is given by the usual right-translation action. Then using Proposition \ref{p:iso} and \eqref{e:strong} we deduce the key isomorphism

\begin{equation}\label{e:classicaladeliciso}
\AA_{\GSp_{2n}(\A)}(K_{2n}(N); \rho, L(\k))^\circ \simeq  \bigoplus_{\substack{1 \le a <N \\ (a, N)=1}} N_\rho(\Gamma_{2n}(N); L(\k))^\circ, \qquad \phi \mapsto (F^{(a)}_\phi)_a.
\end{equation}

Recall that $\rho_\k$ denotes the $K_\infty$-type with highest weight $\k$. Also, we let $\rho_{k}$ be a shorthand for the $K_\infty$-type $\rho_{k, \ldots, k}$ with highest weight $k, \ldots k$ (recall that $k=k_1$). Using \eqref{e:classicaladeliciso} and the notation of section \ref{s:proj} we obtain  the commutative diagram of isomorphisms \begin{equation}\label{e:commutative} \begin{CD} \AA_{\GSp_{2n}(\A)}(K_{2n}(N); \rho_\k, L(\k))^\circ @ > \simeq > {\phi \mapsto (F^{(a)}_\phi)_a} > \bigoplus_{\substack{1 \le a <N \\ (a, N)=1}} S_\k(\Gamma_{2n}(N)) \\ @V{\simeq}V{\widetilde{Y}}V @V{\simeq}V{\oplus_a\D_\k}V \\  \AA_{\GSp_{2n}(\A)}(K_{2n}(N); \rho_{k}, L(\k))^\circ @ > \simeq >{\phi \mapsto (F^{(a)}_\phi)_a}> \bigoplus_{\substack{1 \le a <N \\ (a, N)=1}} N_k(\Gamma_{2n}(N); L(\k))^\circ \end{CD} \end{equation} For brevity, we henceforth use the notation $$V_{N, \k} = \D_\k(S_\k(\Gamma_{2n}(N))) = N_k(\Gamma_{2n}(N); L(\k))^\circ.$$

\begin{lemma}\label{l:compatibilitysp}Let $\widetilde{F}_1=(F_1^{(a)})_a$,  $\widetilde{F}_2=(F_2^{(a)})_a$ be elements of $\bigoplus_{\substack{1 \le a <N \\ (a, N)=1}} N_\rho(\Gamma_{2n}(N); L(\k))^\circ$. Let $\phi_1$, $\phi_2$ be the elements of $\AA_{\GSp_{2n}(\A)}(K_{2n}(N); \rho, L(\k))^\circ$ corresponding to $\widetilde{F}_1$, $\widetilde{F}_2$ respectively  under the isomorphism \eqref{e:classicaladeliciso}. Suppose that $F_1^{(1)} = F_2^{(1)}$. Then there exists an irreducible constituent $\Pi_1\in \mathcal{R}_{\k}(N)$ of the representation of $\GSp_{2n}(\A)$ generated by $\phi_1$, and an irreducible constituent $\Pi_2\in \mathcal{R}_{\k}(N)$ of the representation of $\GSp_{2n}(\A)$ generated by $\phi_2$ such that $\Pi_1 \sim \Pi_2$.
\end{lemma}
\begin{proof} For $i=1,2$, let $\sigma_i$ be the representation of $\GSp_{2n}(\A)$ generated by $\phi_i$. By assumption, $\phi_1|_{\Sp_{2n}(\A)} = \phi_2|_{\Sp_{2n}(\A)}$.   It follows that the restrictions to ${\Sp_{2n}(\A)}$ of $\sigma_1$ and $\sigma_2$ have a common non-zero quotient. So for some irreducible constituents $\Pi_1$, $\Pi_2$ of $\sigma_1$, $\sigma_2$ respectively, there exists an irreducible cuspidal automorphic representation of $\Sp_{2n}(\A)$ that occurs as an \emph{automorphic restriction} (in the sense of \cite[\S5.1]{GeeTaibbi20}) of both $\Pi_1$ and $\Pi_2$ . Hence using \cite[Lemma 5.1.1]{GeeTaibbi20} we see that $\Pi_1\sim \Pi_2$.
\end{proof}

Let  $K'_{2n}(N)$ be the subgroup  of  elements $g \in \GSp_{2n}(\hat{\Z})$ such that $g \equiv \mat{I_n}{}{}{aI_n} \pmod{N}$ for some  $a\in \hat{\Z}$. Given any $F$ in $V_{N, \k}$, let the \emph{adelization} $\Phi_F$ be the function on $\GSp_{2n}(\A)$ defined as $$\Phi_F(g) = \det(J(g_\infty, iI_n))^{-k} F(g_\infty(iI_n)),$$ where we write any element $g \in \GSp_{2n}(\A) $ as
$$
 g = \lambda g_\Q g_\infty k_\mathfrak{f}, \qquad g_\Q\in \GSp_{2n}(\Q), \ g_\infty \in \Sp_{2n}(\R),\ k_\mathfrak{f} \in K'_{2n}(N), \ \lambda \in Z(\R)^+.
$$
It is clear then that $\phi_F \in \AA_{\GSp_{2n}(\A)}(K'_{2n}(N); \rho_k, L(\k))^\circ$. One has the following commutative diagram. \begin{equation}\label{e:commutativeadel}
 \begin{CD} \AA_{\GSp_{2n}(\A)}(K_{2n}(N); \rho_k, L(\k))^\circ @ > \simeq > {\phi \mapsto (F^{(a)}_\phi)_a} > \bigoplus_{\substack{1 \le a \le N \\ (a, N)=1}} N_k(\Gamma_{2n}(N); L(\k))^\circ \\ @A{\iota}AA @AA{F \mapsto (F,F,\ldots,F)}A \\  \AA_{\GSp_{2n}(\A)}(K'_{2n}(N); \rho_{k}, L(\k))^\circ @ > {\substack{\phi_F \mathrel{\reflectbox{\ensuremath{\scriptstyle\mapsto}}} F\\ \simeq}} >{\phi \mapsto F_\phi}>  N_k(\Gamma_{2n}(N); L(\k))^\circ=V_{N, \k}  \end{CD}
\end{equation}

In the above diagram,  the top row coincides with the bottom row of \eqref{e:commutative}, and the map $\iota$ is the inclusion. \emph{Prima facie}, the bottom isomorphism in \eqref{e:commutativeadel} appears to give a cleaner way to go back and forth between classical and adelic forms than the top one; however there is one major disadvantage. Namely, every automorphic representation $\Pi \in \mathcal{R}_{\k}(N)$ is generated by some element in  $\AA_{\GSp_{2n}(\A)}(K_{2n}(N); \rho_k, L(\k))^\circ$ but the same is not known to be true  for $\AA_{\GSp_{2n}(\A)}(K'_{2n}(N); \rho_k, L(\k))^\circ$ (unless $n\le 2$ in which case the latter assertion can be proved identically to \cite[Theorem 2]{sahapet}). In other words, we do not know whether every representation in $\mathcal{R}_{\k}(N)$ is generated via the adelization process. Nonetheless, we will soon observe (see next lemma) that for each $\Pi \in \mathcal{R}_{\k}(N)$, \emph{some} representative $\Pi' \in [\Pi]$  is generated by an element of $\AA_{\GSp_{2n}(\A)}(K'_{2n}(N); \rho_k, L(\k))^\circ$, and hence obtained via adelization.

For each $F \in V_{N, \k}$, let $\Pi_F$ denote the representation of $\GSp_{2n}(\A)$ generated by $\phi_F$. Clearly $\Pi_F$ is a finite direct sum of irreducible cuspidal automorphic representations. We now make the following definition. Given a class $[\Pi] \in \widetilde{\mathcal{R}}_{\k}(N)$, we let  $V_{N, \k}([\Pi])$ be the space generated by all $F \in V_{N, \k}$ with the property that each irreducible constituent of $\Pi_F$ belongs to $[\Pi]$. We define $S_\k(\Gamma_{2n}(N);[\Pi])$ similarly.
 It is clear that $V_{N, \k}([\Pi]) = \D_\k(S_\k(\Gamma_{2n}(N); [\Pi])).$ We have direct sum decompositions \begin{equation}\label{E:VN}V_{N, \k}= \bigoplus_{[\Pi] \in \widetilde{\mathcal{R}}_{\k}(N)} V_{N, \k}([\Pi]) \end{equation} $$S_\k(\Gamma_{2n}(N)) =  \bigoplus_{[\Pi] \in \widetilde{\mathcal{R}}_{\k}(N)} S_\k(\Gamma_{2n}(N); [\Pi])$$ into a sum of orthogonal subspaces with respect to the Petersson inner product.

 \begin{lemma}\label{l:VNKPi} Let $\Pi \in \mathcal{R}_{\k}(N)$.
 \begin{enumerate} \item There exists $F \in V_{N, \k}$ such that $\Pi_F$ is irreducible and satisfies $\Pi_F \sim \Pi$. In particular, $V_{N, \k}([\Pi]) \neq \{0\}$.
 \item The space $V_{N, \k}([\Pi])$ has a basis consisting of forms $F$ as above.
     \end{enumerate}
 \end{lemma}
\begin{proof} Let $\phi \in \AA_{\GSp_{2n}(\A)}(K_{2n}(N); \rho_k, L(\k))^\circ$ be such that $\phi$ generates $\Pi$. Let $F' =F_\phi^{(1)}$.  Using Lemma \ref{l:compatibilitysp}, we see that some irreducible constituent of $\Pi_{F'}$ lies inside $[\Pi]$. Letting $F$ be the projection of $F'$ onto this constituent, we see that $\Pi_F \sim \Pi$. This proves the first assertion.
The second assertion is immediate from the definitions.\end{proof}

Let $[\Pi]\in \widetilde{\mathcal{R}}_{\k}(N)$. By Theorem 4.2.3 of \cite{blaharam} (see also the proof of Theorem 3.13 of \cite{sahapet}), it follows that $$F \in S_\k(\Gamma_{2n}(N);[\Pi]) \;\Longrightarrow\; {}^\sigma F \in S_\k(\Gamma_{2n}(N); {}^\sigma[\Pi]).$$ Using the $\Aut(\C)$-equivariance of the $\D_\k$ map given by Proposition \ref{p:uoper} (see also \eqref{e:commutative}) we obtain \begin{equation}\label{FVNEQ}
 F \in V_{N, \k}([\Pi]) \;\Longrightarrow\; {}^\sigma F \in V_{N, \k}({}^\sigma[\Pi]).
\end{equation}
 In particular, the space $V_{N, \k}([\Pi])$ is preserved under the action of the group $\Aut(\C/\Q([\Pi]))$. Using Lemma 3.17 of \cite{sahapet}, it follows that the space $V_{N, \k}([\Pi])$ has a basis consisting of forms whose Fourier coefficients are in $\Q([\Pi])$.  We note also that given any  irreducible cuspidal automorphic representation $\Pi$ of $\GSp_{2n}(\A)$ such that $\Pi_\infty \simeq L(\k)$, there exists $N$ such that $\Pi \in  \mathcal{R}_{\k}(N)$, and consequently, using Lemma \ref{l:VNKPi} we have $V_{N, \k}([\Pi]) \neq \{0\}.$
\subsection{Eisenstein series}Define the element $
 Q = \left[\begin{smallmatrix}I_n &0&0&0 \\ 0& 0&0&I_n\\ 0&0&I_n&I_n\\ I_n&-I_n&0&0 \end{smallmatrix}\right]
$ embedded diagonally in $\prod_{p<\infty}\Sp_{4n}(\Z_p)$, and for any $\tau \in \hat{\Z}^\times = \prod_{p<\infty}\Z_p^\times$, let $Q_\tau = \mat{\tau I_{2n}}{}{}{I_{2n}}Q \mat{\tau^{-1} I_{2n}}{}{}{I_{2n}}$. Note that $Q = Q_1$. Given a positive integer $N= \prod_p p^{m_p}$, and a primitive Dirichlet character  $\chi$ satisfying $\cond(\chi)\mid N$ and $\chi_\infty = \sgn^{k}$, we define the Eisenstein series  $E_{k,N}^\chi(Z, s; Q)$ and $E_{k,N}^\chi(Z, s; Q_\tau)$ for each
\footnote{As usual, the Eisenstein series is given by an absolutely convergent series for $\Re(s)$ sufficiently large, and by analytic continuation outside that region.}
$s \in \C$ as in \cite[(117)]{PSS17} (see also Section 2.3 of \cite{PSS19}).

Let $0 \le m_0 \le \frac{k}2 - \frac{n + 1}{2}$, and if $\chi^2=1$ assume that $m_0 \neq \frac{k}2-\frac{n+1}2$. By Proposition 6.6 of \cite{PSS17} we have that \begin{equation}\label{e:eisnearholo}E_{k,N}^\chi(\mat{Z_1}{}{}{Z_2}, -m_0; Q_\tau) \in N_k(\Gamma_{2n}(N))\otimes N_k(\Gamma_{2n}(N));\end{equation} furthermore, for $m_0$ as above and $\sigma \in \Aut(\C)$ we have by Proposition 6.8 of \cite{PSS17},
\begin{equation}\label{e:eisarith}
 {}^\sigma \left(\pi^{-2m_0n} E_{k,N}^\chi(Z, -m_0; Q)\right) = \pi^{-2m_0n} E_{k,N}^{{}^\sigma\!\chi}\left(Z, -m_0;Q_\tau\right),
\end{equation}
where $\tau \in \hat{\Z}^\times$ is the element corresponding to $\sigma$ via the natural map $\Aut(\C) \rightarrow \Gal(\Q_{\rm{ab}}/\Q) \simeq \hat{\Z}^\times$.

For $[\Pi] \in \widetilde{\mathcal{R}}_{\k}(N)$, an integer $r$ satisfying $1 \le r \le k_n -n$, $r +n-k \equiv 0 \pmod{2}$, and a primitive Dirichlet character $\chi$ satisfying $\cond(\chi)\mid N$, $\chi_\infty = \sgn^{k}$, $r=1 \Rightarrow \chi^2 \neq 1$, define
\begin{equation}\label{GkNdefeq}
 G_{k,N}^\chi(Z_1, Z_2, r; Q_\tau) := \pi^{n(r+n-k)}  E_{k,N}^\chi(\mat{Z_1}{}{}{Z_2},   \frac{n-k+r}2; Q_\tau),
\end{equation}
\begin{equation}\label{e:defcn}
 C_N([\Pi], \chi, r):=  i^{nk} \pi^{n(r+n -k)} \,  \frac{ \pi^{n(n+1)/2} L^N(r,\Pi \boxtimes \chi)  A_{\mathbf{k}}(r-1) \prod_{p|N}{\rm vol}(\Gamma_{2n}(p^{m_p})) }{\vol(\Sp_{2n}(\Z) \bs \Sp_{2n}(\R))L^N(r+n, \chi) \prod_{j=1}^n L^N(2r+2j-2, \chi^2)},
\end{equation}
where the rational number $A_{\mathbf{k}}(r-1)$ is defined in \cite[(106)]{PSS17}.
Note that $C_N([\Pi], \chi, r)$ depends only on the class $[\Pi]$. Given any $G \in V_{N, \k}([\Pi])$, Corollary 6.5 of \cite{PSS17} tells us that \begin{equation}\label{e:pullbackformulaclassical}
\langle G_{k,N}^\chi( - , Z_2, r; Q_\tau) , \ G \rangle = \chi(\tau)^{-n} C_N([\Pi], \chi, r) \ \bar{G}(Z_2). \end{equation}
We complement this with the following lemma which considers the inner product in the other variable.
\begin{lemma}\label{l:pullbackformulasymm} Let $k, N, \chi, r, \tau$ be as above. For each $G \in V_{N, \k}([\Pi])$, there exists $H \in V_{N, \k}([\Pi])$ such that \begin{equation}\label{e:pullbackformulaclassicalsymm}
\langle G_{k,N}^\chi( Z_1 ,  - , r; Q_\tau) , \ G \rangle =  \bar{H}(Z_1). \end{equation}
\end{lemma}
\begin{proof} Let $\phi$ denote the adelization of $G$. We may assume without loss of generality that $\phi$ generates an irreducible representation $\Pi' \in [\Pi]$. Define the function $H$ as in the lemma. Clearly, $H \in N_k(\Gamma_{2n}(N))$ and $\bar{H}(Z_2) = \overline{H(-\overline{Z_2})}$. Using a standard adelic-to-classical translation (see Theorems 6.1 and 6.6 of \cite{PSS17}), the lemma will follow if we can show that \begin{equation}\label{maindefinition}
\int\limits_{\Sp_{2n}(F) \bs\,g\cdot\Sp_{2n}(\A)} E((
 h  ,g), s, f^{(Q_\tau)}) \phi(g)\,dg \in V_{\Pi'}
\end{equation}
where  $E((
 h  ,g), s, f^{(Q_\tau)})$  is defined as in Section 6.1 of \cite{PSS17}.  However \eqref{maindefinition} follows immediately from Theorem 3.6 of \cite{PSS17} and the observation that  $E((
 h  ,g), s, f^{(Q_\tau)}) =  E((
 g ,h), s, f')$ for some suitable section $f'$ obtained by right-translation of $f$.
 \end{proof}

 We can now prove the following result.
\begin{lemma}\label{l:expansionpullback}Let $k, N, \chi, r, \tau$ be as above. For each $[\Pi']$ in  $\widetilde{\mathcal{R}}_{\k}(N)$, let $\mathfrak{B}_{[\Pi']}$ be an orthogonal basis of $V_{N, \k}([\Pi'])$. Let $W_{N, \k}$ denote the orthogonal complement of $V_{N, \k}$ in $N_k(\Gamma_{2n}(N))$ and let $\mathcal{C}_{N, \k}$ be a basis of $W_{N, \k}$. Then there exists complex numbers $\alpha_{H_1, H_2}$ such that \begin{equation}\begin{split}\label{keyidentity}G_{k,N}^\chi(Z_1, Z_2, r; Q_\tau) = & \chi(\tau)^{-n} \sum_{[\Pi'] \in \widetilde{\mathcal{R}}_{\k}(N)} C_N([\Pi'], \chi, r) \sum_{G \in \mathfrak{B}_{[\Pi']}} \frac{G(Z_1) \bar{G}(Z_2)}{\langle G, G\rangle} \\ &+ \sum_{H_1, H_2 \in \mathcal{C}_{N, \k}} \alpha_{H_1, H_2} H_1(Z_1) H_2(Z_2).\end{split} \end{equation}
\end{lemma}
\begin{proof} Note that $\mathfrak{B}_{N, \k} = \bigcup_{[\Pi'] \in \widetilde{\mathcal{R}}_{\k}(N)} \mathfrak{B}_{[\Pi']}$ is an orthogonal basis of $V_{N, \k}$. Let $\bar{\mathfrak{B}}_{N, \k}$ be the set obtained by replacing each element $G$ of $\mathfrak{B}_{N, \k}$ by the element $\bar{G}$; then using the identity $\langle F_1, F_2 \rangle = \overline{\langle \bar{F_1}, \bar{F_2} \rangle}$, we see that $\bar{\mathfrak{B}}_{N, \k}$  is also an orthogonal basis of $V_{N, \k}$. So $\mathfrak{B}_{N, \k} \cup \mathcal{C}_{N, \k}$ and $\bar{\mathfrak{B}}_{N, \k} \cup \mathcal{C}_{N, \k}$ are both bases of $N_k(\Gamma_{2n}(N))$. From \eqref{e:eisnearholo} and \eqref{e:pullbackformulaclassical} we obtain an expression of the form
\begin{equation}\begin{split}\label{keyidentityprelim}
 G_{k,N}^\chi(Z_1, Z_2, r; Q_\tau) = & \chi(\tau)^{-n}\sum_{[\Pi'] \in \widetilde{\mathcal{R}}_{\k}(N)} C_N([\Pi'], \chi, r) \sum_{G \in \mathfrak{B}_{[\Pi']}} \frac{G(Z_1) \bar{G}(Z_2)}{\langle G, G\rangle} \\
 &+ \sum_{\substack{H_1 \in \mathcal{C}_{N, \k} \\ G_2 \in \mathfrak{B}_{N, \k}}} \alpha_{H_1, G_2} H_1(Z_1)\bar G_2(Z_2) \\
 &+ \sum_{H_1, H_2 \in \mathcal{C}_{N, \k}} \alpha_{H_1, H_2} H_1(Z_1) H_2(Z_2).\end{split} \end{equation} Now using Lemma \ref{l:pullbackformulasymm} we see that each $\alpha_{H_1, G_2}=0$.
\end{proof}

\begin{remark}If $n+2 \le r$, then by the results of \cite{PSS19} we have that $G_{k,N}^\chi(Z_1, Z_2, r; Q_\tau) $ is \emph{cuspidal} in each of the variables $Z_1$, $Z_2$.
\end{remark}
\subsection{Algebraicity of critical \texorpdfstring{$L$}{}-values}
\begin{proposition}\label{mainprop}
 Let $\Pi \in \mathcal{R}_{\k}(N)$, $\chi$ a primitive Dirichlet character such that $\cond(\chi)\mid N$ and $\chi_\infty = \sgn^k$, and $F \in V_{N, \k}([\Pi])$ be such that the Fourier coefficients of $F$ lie in a CM field. Let $r$ be an integer such that $1 \le r \le k_n-n$, $r \equiv k_n-n \pmod{2}$; if $r=1$ assume that $\chi^2 \neq 1$. Then for any  $\sigma \in \Aut(\C)$, we have
 $$
  \sigma \left(\frac{G(\chi)^nC_N([\Pi], \chi, r)}{\langle F, F\rangle} \right) = \frac{G({}^\sigma\!\chi)^nC_N({}^\sigma[\Pi], {}^\sigma\!\chi, r)}{\langle {}^\sigma\! F, {}^\sigma\!F\rangle},
 $$
 where $
  C_N([\Pi], \chi, r)$ is defined in \eqref{e:defcn}.
\end{proposition}

\begin{proof}
The proof follows the method\footnote{We are grateful to the referee for encouraging us to take a closer look at \cite[Appendix]{bocherer-schmidt}.} of the lemma in \cite[Appendix]{bocherer-schmidt}. For each $[\Pi']$ in  $\widetilde{\mathcal{R}}_{\k}(N)$, we pick an orthogonal basis $\mathfrak{B}_{[\Pi']}$ of $V_{N, \k}([\Pi'])$ such that $\mathfrak{B}_{[\Pi]}$ includes $F$ and  $\mathfrak{B}_{ {}^\sigma \![\Pi]}$ includes $ {}^\sigma \!F$. We can do this thanks to \eqref{FVNEQ}. Let $W_{N, \k}$ denote the orthogonal complement of $V_{N, \k}$ in $N_k(\Gamma_{2n}(N))$ and let $\mathcal{C}_{N, \k}$ be a basis of $W_{N, \k}$.

We now use Lemma \ref{l:expansionpullback} to express ${}^\sigma G_{k,N}^\chi(Z_1, Z_2, r; Q)$ in two different ways. First, using \eqref{e:eisarith} and Lemma \ref{l:expansionpullback} we obtain
\begin{equation}\begin{split}\label{keyidentity2}{}^\sigma G_{k,N}^\chi(Z_1, Z_2, r; Q) =& G_{k,N}^{{}^\sigma\! \chi}(Z_1, Z_2, r; Q_\tau)
\\= & {}^\sigma\! \chi(\tau)^{-n} \sum_{[\Pi'] \in \widetilde{\mathcal{R}}_{\k}(N)} C_N([\Pi'], {}^\sigma\! \chi, r) \sum_{G \in \mathfrak{B}_{[\Pi']}} \frac{G(Z_1) \bar{G}(Z_2)}{\langle G, G\rangle} \\ &+ \sum_{H_1, H_2 \in \mathcal{C}_{N, \k}} \alpha'_{H_1, H_2} H_1(Z_1) H_2(Z_2).\end{split} \end{equation} for some $\alpha'_{H_1, H_2} \in \C$. On the other hand, by letting $\sigma$ act on each term of the expression for $G_{k,N}^\chi(Z_1, Z_2, r; Q) = G_{k,N}^\chi(Z_1, Z_2, r; Q_1)$ in the expression given by  Lemma \ref{l:expansionpullback}, we obtain \begin{equation}\begin{split}\label{keyidentity3}{}^\sigma G_{k,N}^\chi(Z_1, Z_2, r; Q) = & \sum_{[\Pi'] \in \widetilde{\mathcal{R}}_{\k}(N)} \sigma(C_N([\Pi'],  \chi, r) ) \sum_{G \in \mathfrak{B}_{[\Pi']}} \frac{{}^\sigma\!G(Z_1) {}^\sigma\!\bar{G}(Z_2)}{\sigma(\langle G, G\rangle)}  \\ &+ \sum_{H_1, H_2 \in \mathcal{C}_{N, \k}} \beta_{H_1, H_2} {}^\sigma \! H_1(Z_1) {}^\sigma \!H_2(Z_2).\end{split} \end{equation} for some $\beta_{H_1, H_2} \in \C$.

We now evaluate $L:=\langle {}^\sigma G_{k,N}^\chi( - , Z_2, r; Q), \ {}^\sigma \!F\rangle$ in two  ways. On the one hand, using \eqref{keyidentity2}, and recalling that ${}^\sigma \!F$ is an element of $\mathfrak{B}_{ {}^\sigma \![\Pi]}$ and that ${}^\sigma \!F$ is orthogonal to all elements in $\left(\bigcup_{\substack{[\Pi']  \in \widetilde{\mathcal{R}}_{\k}(N) \\ [\Pi'] \neq {}^\sigma \![\Pi]}} \mathfrak{B}_{ [\Pi']} \right) \bigcup \mathcal{C}_{N, \k}$,
 we get \begin{equation}\label{e:compare1}L =   {}^\sigma\! \chi(\tau)^{-n} C_N({}^\sigma \![\Pi], {}^\sigma\! \chi, r) \overline{{}^\sigma \!F}(Z_2) = {}^\sigma\! \chi(\tau)^{-n} C_N({}^\sigma \![\Pi], {}^\sigma\! \chi, r) {}^\sigma \!\bar{F}(Z_2)\end{equation} where the second equality uses our hypothesis that the Fourier coefficients of $F$ belong in a CM field.
On the one hand, using \eqref{keyidentity3}, we get \begin{equation}\label{e:compare2}\begin{split}L = & \sum_{[\Pi'] \in \widetilde{\mathcal{R}}_{\k}(N)} \sigma(C_N([\Pi'],  \chi, r) ) \sum_{G \in \mathfrak{B}_{[\Pi']}} \frac{\langle {}^\sigma\!G ,  {}^\sigma \!F \rangle }{\sigma(\langle G, G\rangle)}{}^\sigma\!\bar{G}(Z_2)  \\ &+ \sum_{H_1, H_2 \in \mathcal{C}_{N, \k}} \beta_{H_1, H_2} \langle {}^\sigma \! H_1, {}^\sigma \!F \rangle   {}^\sigma \!H_2(Z_2).\end{split}\end{equation}

We now make the obvious but crucial observation that the elements of the set ${}^\sigma \! \bar{\mathfrak{B}}_{N, \k} \cup {}^\sigma \!\mathcal{C}_{N, \k}$ are linearly independent. So by comparing the coefficients of ${}^\sigma \!\bar{F}(Z_2)$
in \eqref{e:compare1} and \eqref{e:compare2}  together with the well known fact $\sigma(G(\chi)^n) =\,^\sigma\!\chi(\tau^n) G({}^\sigma\! \chi)^n$, we obtain the desired equality.
\end{proof}


In the following theorem, we will obtain an algebraicity result for the special values $L^S(r,\Pi \boxtimes \chi)$ at all the critical points $r$ in the right half plane.

\begin{theorem}\label{t:globalthmarithmetic}
 Let  $k_1\ge k_2 \ge \ldots\ge k_n\ge n+1$ be integers  where all $k_i$ have the same parity. Let $\Pi$ be an irreducible cuspidal automorphic representation of $\GSp_{2n}(\A)$ such that $\Pi_\infty$ is the holomorphic discrete series representation with highest weight $(k_1, k_2, \ldots, k_n)$. Let $S$ be a finite set of places of $\Q$ including $\infty$ such that $\Pi_p$ is unramified for all primes $p \notin S$. Let $F$ be a nearly holomorphic cusp form of  scalar weight $k_1$ (with respect to some congruence subgroup) whose Fourier coefficients lie in a CM field, such that any irreducible constituent $\Pi_F = \otimes_v \Pi_{F,v}$ of the automorphic representation generated by (the adelization) $\Phi_F$ satisfies $\Pi_{F,\infty} \simeq \Pi_\infty$ and $\Pi_{F,p} \sim \Pi_p$ for $p \notin S$. Let $\chi$ be a Dirichlet character such that $\chi_\infty = \sgn^{k_1}$ and let $r$ be an integer such that $1 \le r \le k_n-n$, $r \equiv k_n-n \pmod{2}$; if $r=1$ assume that $\chi^2 \neq 1$.  Then for  $\sigma \in \Aut(\C)$,  we have
 \begin{equation}\label{t:globalthmarithmeticeq1}
  \sigma \left(\frac{L^S(r,\Pi \boxtimes \chi)}{i^{k_1}\pi^{nk_1+nr+r} G(\chi)^{n+1} \langle F, F\rangle}\right) = \frac{L^S(r,{}^\sigma\Pi \boxtimes {}^\sigma\!\chi)}{i^{k_1}\pi^{nk_1+nr+r} G({}^\sigma\!\chi)^{n+1}\langle {}^\sigma\!F,{}^\sigma\!F\rangle}.
 \end{equation}
\end{theorem}
\begin{proof}  We fix an integer $N$ divisible by all primes in $S$ such that $\cond(\chi)\mid N$ and  $F \in V_{N, \k}([\Pi])$.
By Proposition \ref{mainprop},
\begin{equation}\label{finaleq1}
 \begin{split}
  &\sigma \left( i^{nk_1}\frac{G(\chi)^nL^N(r,\Pi \boxtimes \chi)}{\pi^{n(k_1-r-n)} L^N(r+n, \chi) \prod_{j=1}^n L^N(2r+2j-2, \chi^2)\langle F, F\rangle} \right) \\ &= i^{nk_1}\frac{G({}^\sigma\!\chi)^nL^N(r,{}^\sigma\Pi \boxtimes {}^\sigma\!\chi)}{\pi^{n(k_1-r-n)} L^N(r+n, {}^\sigma\!\chi) \prod_{j=1}^n L^N(2r+2j-2, {}^\sigma\!\chi^2)\langle {}^\sigma\!F, {}^\sigma\!F\rangle}.
 \end{split}
\end{equation}

If $p$ is a prime such that $p \notin S$, then it is clear from the definition of local $L$-factors that \begin{equation}\label{e:aux}
 \sigma \left(\frac{L(r,\Pi_p \boxtimes \chi_p)}{L(r+n, \chi_p) \prod_{j=1}^n L(2r+2j-2, \chi_p^2)} \right) \\ = \frac{L(r,{}^\sigma\Pi_p \boxtimes {}^\sigma\!\chi_p)}{ L(r+n, {}^\sigma\!\chi_p) \prod_{j=1}^n L(2r+2j-2, {}^\sigma\!\chi_p^2)}.
\end{equation}

Combining \eqref{finaleq1} and \eqref{e:aux} we get \begin{equation}\label{finaleq2}
 \begin{split}
  &\sigma \left( i^{nk_1}\frac{G(\chi)^nL^S(r,\Pi \boxtimes \chi)}{\pi^{n(k_1-r-n)} L^S(r+n, \chi) \prod_{j=1}^n L^S(2r+2j-2, \chi^2)\langle F, F\rangle} \right) \\ &= i^{nk_1}\frac{G({}^\sigma\!\chi)^nL^S(r,{}^\sigma\Pi \boxtimes {}^\sigma\!\chi)}{\pi^{n(k_1-r-n)} L^S(r+n, {}^\sigma\!\chi) \prod_{j=1}^n L^S(2r+2j-2, {}^\sigma\!\chi^2)\langle {}^\sigma\!F, {}^\sigma\!F\rangle}.
 \end{split}
\end{equation}

For a Dirichlet character $\psi$ and a positive integer $t$ satisfying $\psi_\infty = \sgn^t$, we have by \cite[Lemma~5]{shi76})
\begin{equation}\label{finaleq3}
 \sigma \left(\frac{L^S(t, \psi)}{(\pi i)^t G(\psi)} \right) = \frac{L^S(t, {}^\sigma\!\psi)}{(\pi i)^t G({}^\sigma\!\psi)}.
\end{equation}
Plugging \eqref{finaleq3} (for $\psi=\chi$ and $\psi=\chi^2$) into \eqref{finaleq2}, we obtain \eqref{t:globalthmarithmeticeq1}.
\end{proof}

\begin{remark}One can choose $F$ in Theorem \ref{t:globalthmarithmetic} such that its Fourier coefficients lie in $\Q([\Pi])$ (see the last paragraph of Section \ref{s:classicaladelic}). For such an $F$, Theorem \ref{t:globalthmarithmetic} implies that $$\frac{L^S(r,\Pi \boxtimes \chi)}{i^{k_1}\pi^{nk_1+nr+r} G(\chi)^{n+1} \langle F, F\rangle} \in \Q([\Pi])\Q(\chi).$$
\end{remark}

\emph{Proof of Corollary \ref{cor:mainintro}.} This follows immediately from Theorem \ref{t:globalthmarithmetic}, and the fact that the Gauss sums $G(\chi)$ are algebraic numbers. In fact, using the algebraic properties of Gauss sums (see Lemma 8 of \cite{shi76}), one observes that the quantity in \eqref{e:last} lies in the CM field given by $\Q([\Pi])\Q(\chi)\Q(e^{\frac{2 \pi i}{\mathrm{cond}(\chi)}})$, where $\chi=\chi_1 \overline{\chi_2}$.

\addcontentsline{toc}{section}{Bibliography}
\bibliography{special-degree-n}{}
\end{document}